%% file: domino2016-08-25Arxiv.tex
\documentclass[11pt,A4paper]{article}

\usepackage[left=32mm,right=32mm,top=30mm,bottom=32mm]{geometry}
\usepackage{epsfig}
\usepackage{epstopdf}

\usepackage{xfrac}

\usepackage{color}
\usepackage{amsthm,amsmath,amssymb}
\usepackage{booktabs}
\usepackage{mathpazo}
\usepackage{microtype}
\usepackage{overpic}
\usepackage{bm}
\usepackage{sectsty}
\usepackage[
	pdftitle={PDFTitle},
	pdfauthor={Hugo Parlier},
	ocgcolorlinks,
	linkcolor=linkred,
	citecolor=linkred,
	urlcolor=linkblue]
{hyperref}

\usepackage{graphicx}
\usepackage{subfig}
\usepackage{pgf,tikz}%for imported figure in tikz
\usetikzlibrary{arrows}%tikz stuff

\definecolor{linkred}{rgb}{0.48,0.1,0.05}
\definecolor{linkblue}{RGB}{16, 78, 139}

\definecolor{leafgreen}{rgb}{0.21,0.66,0.24}

\usepackage[hang,flushmargin]{footmisc}
\usepackage{enumitem}
\usepackage{titlesec}
	\titlespacing{\section}{0pt}{12pt}{0pt}
	\titlespacing{\subsection}{0pt}{6pt}{0pt}
	
\titlelabel{\thetitle.\quad}

\makeatletter 

\long\def\@footnotetext#1{% 
\H@@footnotetext{% 
\ifHy@nesting 
\hyper@@anchor{\@currentHref}{#1}% 
\else 
\Hy@raisedlink{\hyper@@anchor{\@currentHref}{\relax}}#1% 
\fi 
}}

\def\@footnotemark{% 
\leavevmode 
\ifhmode\edef\@x@sf{\the\spacefactor}\nobreak\fi 
\H@refstepcounter{Hfootnote}% 
\hyper@makecurrent{Hfootnote}% 
\hyper@linkstart{link}{\@currentHref}% 
\@makefnmark 
\hyper@linkend 
\ifhmode\spacefactor\@x@sf\fi 
\relax 
}% 

\ifFN@multiplefootnote% 
\renewcommand*\@footnotemark{% 
\leavevmode 
\ifhmode 
\edef\@x@sf{\the\spacefactor}% 
\FN@mf@check 
\nobreak 
\fi 
\H@refstepcounter{Hfootnote}% 
\hyper@makecurrent{Hfootnote}% 
\hyper@linkstart{link}{\@currentHref}% 
\@makefnmark 
\hyper@linkend 
\ifFN@pp@towrite 
\FN@pp@writetemp 
\FN@pp@towritefalse 
\fi 
\FN@mf@prepare 
\ifhmode\spacefactor\@x@sf\fi 
\relax% 
}% 
\fi 

\makeatother 

\theoremstyle{plain}
\newtheorem{theorem}{Theorem}[section]

\theoremstyle{definition}

\newcommand{\R}{{\mathbb R}}

\newcommand{\CC}{{\mathcal C}}

\newcommand{\diam}{{\rm diam}}
\newcommand{\lev}{{\rm lev}}

\newcommand{\FF}{\mathcal F}

\newcommand{\vol}{{\rm vol}}

\newtheorem*{thmA}{Theorem A}
\newtheorem*{thmB}{Theorem B}
\newtheorem*{thmC}{Theorem C}

\sectionfont{\large \sc}
\subsectionfont{\normalsize}

\long\def\symbolfootnote[#1]#2{\begingroup%
\def\thefootnote{\fnsymbol{footnote}}\footnote[#1]{#2}\endgroup}

\def\blfootnote{\xdef\@thefnmark{}\@footnotetext}

\begin{document}

{\Large \bfseries \sc Distances in domino flip graphs}

%{\Large \bfseries \sc on hyperbolic surfaces}

{\bfseries Hugo Parlier\symbolfootnote[1]{\normalsize Research supported by Swiss National Science Foundation grant number PP00P2\textunderscore 153024 
} and Samuel Zappa}

{\em Abstract.} 
This article is about measuring and visualizing distances between domino tilings. Given two tilings of a simply connected square tiled surface, we're interested in the minimum number of flips between two tilings. Given a certain shape, we're interested in computing the diameters of the flip graphs, meaning the maximal distance between any two of its tilings. Building on work of Thurston and others, we give geometric interpretations of distances which result in formulas for the diameters of the flip graphs of rectangles or Aztec diamonds.
\vspace{1cm}

\section{Introduction}\label{sec:introduction}

Let $S$ be a square tiled surface by which we mean a surface obtained by pasting together $1$ by $1$ Euclidean squares. We'll generally be interested in when $S$ is a simply connected shape cut out from standard square tiling of the plane. A good example to keep in mind is when $S$ is an $n$ by $m$ rectangle. A {\it domino tiling} of $S$ is a tiling of $S$ by $2$ by $1$ rectangles (dominos). Note that even if $S$ is made of an even number of squares, it might not be tileable. An example is given by the infamous mutilated chessboard, an $n$ by $n$ board with two opposite corners removed. We're interested in understanding the set of all tilings of $S$ when they exist.

If a tiling $T$ of $S$ has two dominos that share a long edge, they fill a $2\times 2$ square and one can obtain a new tiling $T'$ of $S$ by rotating the square by a quarter turn. This operation we call a {\it flip} (see \autoref{ghx:intro}). 

\input{intro}

Dual to the square tiling of $S$ is a graph $S^{*}$ and domino tilings of $S$ are easily represented in $S^{*}$ as collections of disjoint edges that cover all vertices (see  \autoref{ghx:intro} for an example). 

With this in mind, associated to tilings of a tileable $S$ is the domino flip graph $\FF_S$, defined as follows. Vertices of $\FF_S$ are tilings and we place an edge between two tilings if they are related by a single flip. We think of the graph $\FF_S$ as a metric space by assigning length $1$ to each edge and we denote the induced distance on the vertices of $\FF_S$ by $d_{\FF_S}$. This gives natural metric on the space of tilings of $S$. We are interested in the geometry of $\FF_S$.

To illustrate the type of questions we're interested in, consider the example of when $S$ is a $n \times 2$ rectangle with $n\geq 1$. The number of tilings of $S$ is the $(n+1)$th Fibonacci number. This is a well known puzzle/exercise that can be proved by an induction argument. A geometric property of these graphs that we'll pay close attention to is their {\it diameter} by which we mean the maximal distance between any two domino tilings. In this case, it's not too difficult to work out.

\input{rectangle1}

We begin with an upper bound on the diameter. For simplicity, suppose that $n$ is even but the general argument is identical. Among all tilings, there is one that stands out: the tiling where all dominos are upright (\autoref{ghx:rectangle1} (b)). Now observe that any two tilings can be joined by a path that passes through this tiling. To construct such a path is easy: if a tiling has any dominos that {\it aren't} upright, they must come in pairs of flippable horizontal dominos. For a given tiling, at most $\frac{n}{2}$ flips are required to put all of the dominos in upright position. In particular that means there is a path of length at most $\frac{n}{2}+\frac{n}{2}=n$ between any two tilings. Now to actually require $n$ flips would mean that all $n$ dominos on both tilings were in horizontal position, but there is only one such tiling, the tiling illustrated on the left in \autoref{ghx:rectangle}. So the two tilings were identical to begin with and were at distance $0$. That allows us to improve the upper bound: any two tilings are at distance at most $n-1$. Perhaps surprisingly, this new upper bound is sharp. 

\input{rectangle2}

The two tilings illustrated above are realize the bound. Indeed, to get from the left tiling to the right one, it will be necessary to flip all of the dominos. That will require $\frac{n}{2}$ flips in total and in particular they will all be in upright position at one point. Now to reach the right hand tiling, $n-2$ of them will have to put back in a horizontal position, which will require an additional $\frac{n-2}{2}$ flips. All in all, any path between them contains last least $n-1$ flips.

Of course arbitrary shapes won't have nice formulas for their diameters like that, but what about other shapes? How does one compute the diameter of the flip graph of the $n \times m$ rectangles? 

Before getting into our results, we observe that some of the questions we ask are similar in spirit to questions that have been investigated for triangulations of surfaces. Given a polygon, the set of its triangulations has a similar structure: one moves between triangulations by flipping edges in the triangulation. The number of triangulations of a polygon is the $n-2$th Catalan number. The associated flip graph has been extensively studied, namely by Sleator, Tarjan and Thurston, who found sharp bounds on the diameter \cite{STT}. Recently, Pournin \cite{Pournin} sharpened their result and  produced explicit examples of triangulations at maximal distance. In fact, the example we give above for the $n$ by $2$ rectangle illustrates, in a much simpler form of course, Pournin's examples. Again an example of a configuration space where the size (number of vertices) and the diameter are elegant quantities. As such, it portrays our point of view quite well and in particular why we are viewing our graphs as a type of moduli space.

Before getting into the geometry of these flip graphs, what about it's topology? In particular, are tilings always related by a sequence of flips? A remarkable theorem, which can be deduced from ideas of Thurston \cite{Thurston, EKLP1, SaldanhaEtAl}, says that if $S$ is simply connected, then $\FF_S$ {\it is} connected. This elegant relationship between the topologies of $S$ and $\FF_{S}$ is not a priori obvious and can be showed using Thurston's height function which we'll describe later. Note there are simple examples of non-simply connected surfaces whose flip graph is disconnected, see Section \ref{sec:distances}. 

One of the main tools we'll be using is the observation (see for instance \cite{SaldanhaEtAl}) which is that associated to an ordered pair of tilings $T,T'$, one obtains a collection of disjoint oriented cycles $\CC:=\CC(T,T')$ in $S^{*}$. We'll give details on why its true in Section \ref{sec:distances}. Using these cycles we can define a value function $\nu$ defined on the vertices $V(S)$ of $S$ 
$$
\nu(v) := | \nu^{+}(v) - \nu^{-}(v) |
$$
where
$$
\nu^{+}(v) := {\text{number of positive cycles of }\CC(T,T'){\text{ surrounding }}v}
$$
and
$$
\nu^{-}(v) := {\text{number of negative cycles of }\CC(T,T'){\text{ surrounding }}v}
$$

Our first interpretation of distance is the following, which relies heavily on a distance formula by Saldanha, Tomei, Casarin and Romualdo \cite{SaldanhaEtAl}, which in turn uses Thurston's height function. 
\begin{thmA}
The distance between $T$ and $T'$ is given by the formula 
$$d_{\FF_S}(T,T') = \sum_{v \in V(S)} \nu(v)$$
\end{thmA}
The advantage of this formula is that it allows another geometric interpretation of distances. In fact, we associate to $\CC(T,T')$ a $3$-dimensional shape constructed as follows. We think of $S$ as a subset of $\R^2$ and thus living in $\R^3$. (Strictly speaking this may not be true if $S$ in not geometrically embeddable in $\R^2$ - but its a useful picture to keep in mind.)  Now in any order, construct the following. To each {\it positive} cycle construct the $1$-thick volume above it. To each {\it negative} cycle, dig a $1$-thick hole below it. We think of the "holes" below $S$ as negative volume. The resulting shape we call the {\it filling shape} $F$ associated to $T$ and $T'$. Note that $F = F^{+} \cup F^{-}$ where $F^{+}$ is a volume lying above $S$ and $F^{-}$ is hole below. An immediate consequence of Theorem A is the following.

\begin{thmB}
Let $F = F^{+} \cup F^{-}$ be the filling shape associated to $T$ and $T'$. Then
$$
d_{\FF_S}(T,T') = \vol(F^{+}) - \vol( F^{-})
$$
\end{thmB}

As in our $n\times 2$ rectangle example above, we'd like to compute diameters $\diam(\FF_S)$ for certain natural shapes $S$ such as when $S$ is a rectangle or an Aztec diamond. These are examples of particular kinds of $S$ that we call {\it Saturnian} because they can be constructed as collection of rings of $1\times 1$ squares. By using the techniques that go into Theorem B, we're able to obtain an expression for the diameters of Saturnian shapes which gives in particular the following theorem.

\begin{thmC}
When $S$ is a $n\times n$ square (with $n$ even):
  	\begin{equation*}
	\diam(\FF_{S}) = \frac{n^{3}-n}{6}
	\end{equation*}
When $S$ is an $m\times n$ rectangle (with $m \geqslant n$ and at least one is even): 
  	\begin{equation*}
	\diam(\FF_{S}) = \left\{\begin{array}{ll}
		\frac{mn^{2}}{4} - \frac{n^{3}}{12} - \frac{n}{6} & \text{if $n$ is even}\\
		\frac{mn^{2}}{4} - \frac{n^{3}}{12} + \frac{n}{12} - \frac{m}{4} & \text{otherwise}
		\end{array}
		\right.
	\end{equation*}
When $S$ is an Aztec diamond of order $n$:
	\begin{equation*}
	\diam(\FF_{S}) = \frac{n^{3}}{3} + \frac{n^{2}}{2} +  \frac{n}{6}
	\end{equation*}
\end{thmC}

We note that the cardinalities of the number of vertices of $\FF_S$ for these shapes have already been studied extensively. The number of tilings of an $m\times n$ rectangle is given by a spectacular exact formula:
$$
\sum_{j=1}^{\lceil \frac{m}{2} \rceil} \sum_{k=1}^{\lceil \frac{n}{2} \rceil} 4\left(\cos^2\frac{\pi j}{m+1} + \cos^2\frac{\pi k}{n+1}\right)
$$
due independently to Kasteleyn and Temperley-Fischer \cite{Kasteleyn, Temperley-Fisher}. The number of tilings of the Aztec diamond is 
$$2^{n(n+1)/2}$$
This result, due to Elkies, Kuperberg, Larsen, Michael and Propp \cite{EKLP1} is referred to as the Aztec diamond theorem. Using the moduli space analogy, these results constitute the computations of the size whereas our results are about the shape.

{\bf Acknowlegements.}
The authors are grateful to B\'eatrice de Tili\`ere for many interesting domino conversations.

\section{Computing distances}\label{sec:distances}

As described above, $S$ is a surface obtained by gluing $1$ by $1$ squares. As such, its vertices $V(S)$ are the points that are the images of the vertices of the squares under the pasting. These are not to be confused with the vertices of $S^{*}$, the dual graph to the pasting. Vertices of $S^{*}$ are the squares that form $S$ and two squares share an edge if the corresponding squares are adjacent. We think of $S^{*}$ as geometrically embedded with the vertices being represented as the centers of the squares (see \autoref{ghx:intro}). Tilings are now in $1$ to $1$ correspondence with {\it perfect matchings} of $S^{*}$: these are collections of edges of $S^{*}$ so that every vertex of $S^{*}$ belongs to exactly one edge.

We're mainly interested in when $S$ is simply connected embedded subset of the usual square tiling of the Euclidean plane. However, the results follow for a more general setup. We ask that $S$ is simply connected and that any interior vertex of $S$ (coming from the pasting of the squares) be surrounded by exactly $4$ squares. Said otherwise, we ask that every interior point of $S$ be locally Euclidean (of curvature $0$). This will be necessary in order to inherit a coloring from a standard black/white coloring of the square tiled Euclidean plane. Observe that $S$ may not necessarily be geometrically embeddable in $\R^2$; an example is given in  \autoref{ghx:surface}. Nonetheless, it might sometimes be convenient to think of $S$ as lying inside the plane $z=0$ inside $\R^3$.

 \input{surface}

When $S$ is domino tileable (it can be tiled by $2$ by $1$ rectangles) we denote $\FF_{S}$ the flip graph of $S$. As mentioned above, $S$ being simply connected implies that $\FF_S$ is connected. A simple example of a non-simply connected $S$ with disconnected $\FF_S$ is given by a $3$ by $3$ square with the middle unit square removed. There are only two possible domino tilings of it, clearly not related by a flip, and so the flip graph in this case consists of two isolated vertices. A much less obvious result, which can be deduced from \cite{SaldanhaEtAl}, is that a $2n+1$ by $2n+1$ square with the middle unit square removed has $n+1$ connected components.

\subsection{Thurston's height function and distance formulas}
In \cite{Thurston}, Thurston described a function which turned out to be quite useful in understanding these flip graphs and similarly structured relatives. We briefly describe it in our context to keep our article as self-contained as possible. Given a tiling $T$, it attributes to each vertex $v$ of $S$ a height $h_{T}(v)$. 

We begin by coloring the squares of $S$ like those of a chessboard. The existence of such a coloring is immediate if $S$ is embedded in the plane, but otherwise it can either be sequentially colored from a given base square or can be colored via an immersion in a chess colored plane. We orient edges of $S$ so that they run clockwise around black squares and counterclockwise around whites. (Equivalently, edges are oriented so that the black square is to their right.) We then choose a boundary vertex $v_{0}$ and set $h_{T}(v_{0}) = 0$. 

\input{height}

The height of a vertex $v$ is then defined as follows. We begin by finding a path between $v_0$ and $v$ that does not cross $T$ (a sequence of edges $(e_{1},\hdots,e_{n})$ not covered by the dominos of $T$). We then define
$$
h_{T}(v):=\sum_{i=1}^{n}o(e_{i}),
$$
where 
$$o(e_{i})=\left\{\begin{array}{ll}
		+1 & \text{if the orientation of $e_{i}$ corresponds to that of the path}\\
		-1 & \text{otherwise}
		\end{array}
		\right.
$$
The function is well defined as, perhaps surprisingly at first, it doesn't depend on the choice of path. Moreover, associated to a height function is a unique tiling and so height functions and tilings are in $1$ to $1$ correspondence. This allows to define a partial order on tilings: $T\leq T'$ if $h_{T}(v) \leq h_{T'}(v)$ for every vertex $v$ of $S$, giving $\FF_S$ the structure of a {\it distributive lattice}. 

Observe that a flip will only modifies the height of its central vertex (the value will change by $4$). A key result \cite{Remila,Thurston} states that $T \leq T'$ if and only if there exists a sequence of flips transforming $T$ into $T'$ and while always {\it increasing} the height of the vertices. Now using the lattice structure, given tilings $T$ and $T'$, there exists a supremum tiling $T \land T'$. This gives us a natural path in $\FF_{S}$ between $T$ and $T'$ by combining the height increasing path between $T$ and $T \land T'$ and the height decreasing path between $T \land T'$ and $T'$. 

The resulting path is a geodesic and thus as a consequence, we get the following distance formula. 
\begin{theorem}[Theorem 3.2 of \cite{SaldanhaEtAl}]\label{thm:heightdistance}
$$
d_{\FF_S}(T,T')=\frac{1}{4}\sum_{v \in V(S)}|h_{T}(v) - h_{T'}(v)|
$$
\end{theorem}
We'll use this formula to give a geometric interpretation of distances below.

\subsection{Cycles associated to tilings}

We fix $S$ and a black/white coloring of its squares, or equivalently, of the vertices of $S^{*}$. We look at the set of cycles $c$ of $S^{*}$ such that the complementary region $(S^{*}\backslash c)^{*}$ is tileable. (The set $(S^{*}\backslash c)^{*}$ is simply $S$ with the squares that $c$ passes through removed.)

We'll build cycles by considering tilings represented by disjoint edges in $S^{*}$ and completing the edges to form a cycle: we'll call these domino cycles. When a cycle is built upon a tiling, one out of every two edges is a domino edge. If a cycle is given an orientation, it's easy to see that all domino edges will begin on the same color. This observation can be used to give domino cycles a natural orientation. We orient dominos from black to white and this gives the cycle an orientation. Using the natural orientation of the plane (counter clockwise is positive), this allows us to distinguish between {\it positive} and {\it negative} domino cycles. A \textit{cycle collection} $\CC$ is a disjoint set of (oriented) cycles.

We'll now use cycles to compute the distance between tilings. We consider an {\it ordered} pair $T,T'$ of tilings. We draw both tilings simultaneously on $S^{*}$, erasing all perfectly superimposed tiles.

{\it Claim:} The union of all non superimposed tiles consists in collection of cycles $\CC(T,T')$, each cycle consisting of edges that alternatively belong to $T$ and $T'$. 

\begin{proof}[Proof of Claim:]
Unless $T = T'$, there are vertices of $S^{*}$ not covered by superimposed dominos. Consider such a vertex. Now there must be exactly one domino of both $T$ and $T'$ in the vertex. As such the subgraph of $S^{*}$ formed by all non superimposed edges of $T$ and $T'$ is a finite subgraph of degree $2$ in every edge. The claim follows.
\end{proof}

We now orient the cycles of  $\CC(T,T')$ using the orientation given by the dominos of $T$ (hence the importance of the order).

As described in the introduction, we define the value function $\nu$ on vertices $v$ of $S$:
$$
\nu(v) := | \nu^{+}(v) - \nu^{-}(v) |
$$
where $\nu^{+}(v)$, resp. $\nu^{-}(v)$, are the number of positive, resp. negative, cycles surrounding $v$. 

We can now interpret distances in terms of cycles.
\begin{thmA}
The distance between $T$ and $T'$ is given by the formula 
$$d_{\FF_S}(T,T') = \sum_{v \in V(S)} \nu(v)$$
\end{thmA}

\begin{proof}
With the help of the distance formula from Theorem \ref{thm:heightdistance}, it suffices to show that for every vertex $v$, $|h_{T}(v) - h_{T'}(v)| = 4 \nu(v)$.

Height functions always coincide on the boundary of $S$ so we need to check the above formula for a vertex $v$ inside $S$. To do so, consider an oriented path from a boundary vertex $v_b$ to $v$ which follows only positively oriented edges of $S$. To construct such a path, consider any edge path between $v_b$ and $v$, and if any of the edges are oriented in the negative direction, they can be replaced by a $3$ edge detour of positively oriented edges.

As we evolve along this path, we're going to play close attention to how $h_T$ and $h_T'$ evolve when we cross cycles. Before doing so we observe the following.

Consider a cycle $c$ of $\CC(T,T')$. There are natural {\it inside} and {\it outside} regions of $S \setminus c$. Perpendicular to the edges of $c$ are the oriented edges of $S$, oriented as in the definition of the height function (black is on their right). When you follow the edges of $c$, the edges of $S$ encountered alternate between pointing inside the cycle and out. Since the edges of a cycle alternate between corresponding to dominos of $T$ and $T'$, we also have the following. If the cycle is positive, dominos of $T$ cover all the exiting edges of $S$ and dominos of $T'$ cover the entering ones. The opposite situation occurs for negative cycles.

Suppose the edge $\vec{e}=(v',v'')$ {\it enters} a positive cycle (there is one more positive cycle surrounding $v''$ than $v'$, i.e., $\nu^{+}(v'')= \nu^{+}(v')+1$). 

Then, the edge $\vec{e}$ is not covered by a domino of $T$. Thus:
$$
h_T(v'') = h_T(v') + 1
$$
However, $\vec{e}$ is covered by a domino of $T'$. To contour this domino, there is a $3$ edge path of negatively oriented edges and thus
$$
h_{T'}(v'') = h_{T'}(v') - 3
$$
So entering a positively cycle changes the difference $h_T - h_{T'}$ by $+4$.

The same argument shows that both entering a negatively oriented cycle or exiting a positively oriented cycle affect $h_T - h_{T'}$ by $-4$. And as one might expect, exiting a negatively oriented cycle changes the difference $h_T - h_{T'}$ by $+4$. 

All in all, for a vertex $v$, we've shown that
$$
h_T(v) - h_{T'}(v) = 4 \nu^{+}(v) - 4 \nu^{-}(v)
$$
and hence 
$$
|h_{T}(v) - h_{T'}(v)| = 4 \nu(v)
$$
as claimed.
\end{proof}

For an example application of the theorem, see \autoref{ghx:distance2}(c). It becomes  straightforward to compute the distance between the two tilings (in this example $54$) by counting cycles surrounding vertices.

We're now going to interpret distance in terms of a filling shape as described in the introduction. 

\subsection{Filling shapes}

We associate to $\CC(T,T')$ a $3$-dimensional shape, subset of $S \times \R$. When $S$ is a subset of $\R^2$, the shape belongs to $\R^3$. The notion of being {\it above} and {\it below} is all relative to $S$. For instance in $\R^3$, a point "above" $S$ is a point with the same $x,y$ coordinates as a point of $S$ but a positive $z$ coordinate. 

We choose any order on the cycles of $\CC(T,T')$ and for each one we perform the following construction. If the cycle is positively oriented, we construct the $1$-thick volume above it. If it is negatively oriented, we dig a $1$-thick hole below it. The resulting shape is a collection of "buildings" and "holes" and we think of the holes below $S$ as being of negative volume. This is the {\it filling shape} $F= F^{+} \cup F^{-}$ associated to $T$ and $T'$ where $F^{+}$ is a volume lying above $S$ and $F^{-}$ is hole below. Notice that exchanging $T$ and $T'$ reflects $F$ through the plane containing $S$.
\input{distance}
An example of two tilings such that their associated filling shape is entirely above $S$ is given in \autoref{ghx:distance}. A more complicated example with non empty $F^{+}$ and $F^{-}$ is given in \autoref{ghx:distance2}.
\input{distance2}

Observe that a filling shape can be built using $1\times 1 \times 1$ cubes set on or below vertices of $S$ (and {\it not} of $S^{*}$).

The filling shape is of interest to us because its volume embodies the distance between the tilings. The following is now a direct consequence of Theorem A.

\begin{thmB}
Any $T,T'\in \FF_{S}$ with filling shape $F= F^{+} \cup F^{-}$ satisfy
$$
d_{\FF_{S}} (T,T') = \vol(F^{+}) - \vol(F^{-})
$$
\end{thmB}

\section{Diameters of flip graphs}\label{sec:diameters}

Let us now focus on the diameters of $\FF_S$ for different surfaces $S$.

The lattice structure implies the existence of a unique maximal element $T^{+}$ and a unique minimal element $T^{-}$. Given any tiling $T$ and any vertex $v$, we have
$$h_{T^{-}}(v) \leqslant h_{T}(v) \leqslant h_{T^{+}}(v)$$
with at least strict inequality for some vertex $v$ (different tilings have different height functions). Now using the distance formula of Theorem \ref{thm:heightdistance}, $T^{+}$ and $T^{-}$ are the unique tilings that realize the diameter of $\FF_S$.

In addition, it follows from Theorem B that $\diam(\FF_S)$ is the maximal volume of filling shapes. Note that if $\CC(T,T')$ is a set of cycles realizing the maximal volume, all of the cycles must have the same orientation, otherwise reorienting cycles in the same way will give a larger volume. By reversing the order of the two diameter realizing tilings if necessary, we can thus suppose that the diameter is realized by a collection of positive cycles.  

We denote by $\vol_{+1}(c)$ for the 1-thick volume built upon a cycle $c$ and we have the following.
\begin{theorem}\label{thm:cyclediameter}
$\diam(\FF_{S}) = \max_{\CC \text{ of $S$}} \sum_{c \in \CC} \vol_{+1}(c)$
\end{theorem}
\input{diameter}
The maximal filling shape seems to be related to a type of isoperimetric profile of $S$. In general, it may be difficult to find it. For example, to find the maximal filling shape illustrated in \autoref{ghx:diameter}, one could either look at Thurston's height function or find local arguments based on colorings. However, for certain types of surfaces, we can exhibit explicit formulas.

To do so we define the {\it first ring} $R_{1}$ to be the set of all $1\times 1$ squares of $S$ that the boundary of $S$ belongs to. We define $S^{1}$ to be $S\setminus R_1$. We then define the rings $R_i$ iteratively: $R_{2}$ is the set of squares of $S^{1}$ that the boundary of $S^{1}$ belongs to. Generally, $R_{i}$ is the set of squares of $S^{i-1}$ that contain $\partial S_{i-1}$ (where $S^{0} = S$). We'll denote by $V_{i}$ the set of vertices that belong to $R^{i}$ but not to $S^{i-1}$.

It will be convenient to define a function on vertices of $S$ as follows: the level $\lev(s)$ of $s \in V(S)$ is the number of 1 by 1 squares of $S$ needed to connect $s$ to the boundary of $S$. So vertices on the boundary of $S$ are of level $0$ and those of level $i$ are exactly the vertices $V_i$.

With this in mind, we say a surface is {\it Saturnian} if each of its rings $R_i$ corresponds to a cycle in $S^{*}$ (and $S\setminus R_i$ is tileable). For these surfaces, the following holds.

\begin{theorem}\label{thm:saturndiameter}
If $S$ is Saturnian, then
$$\diam(\FF_{S}) = \sum_{v \in V(S)} \lev(v) = \sum_{i\geq 1} i |V_{i}|$$
where $|V_{i}|$ is the cardinality of $V_{i}$.
\end{theorem}

\begin{proof}
The volume of a filling shape is positive and can be computed by summing the number of $1\times 1 \times 1$ blocks beneath the vertices of $S$. For a vertex $v$, this number cannot be any more than its level. The set of vertices $V_i$ are those at exactly distance $i$ from the boundary, and so we get the upper bound of 
$$
\sum_{v \in V(S)} \lev(v)=\sum_{i\geq 1} i |V_{i}|
$$
Now if $S$ is Saturnian, then the natural cycle decomposition of the rings gives the same lower bound.
\end{proof}

For standard surfaces that are Saturnian, we can compute this formula explicitly.
\begin{thmC}
When $S$ is a $n\times n$ square $Q(n)$, with $n$ even:
  	\begin{equation*}
	\diam(\FF_{Q(n)}) = \frac{n^{3}-n}{6}
	\end{equation*}
When $S$ is an $m\times n$ rectangle $R(m,n)$, with $m \geqslant n$ and at least one is even: 
  	\begin{equation*}
	\diam(\FF_{R(m,n)}) = \left\{\begin{array}{ll}
		\frac{mn^{2}}{4} - \frac{n^{3}}{12} - \frac{n}{6} & \text{if $n$ is even}\\
		\frac{mn^{2}}{4} - \frac{n^{3}}{12} + \frac{n}{12} - \frac{m}{4} & \text{otherwise}
		\end{array}
		\right.
	\end{equation*}
When $S$ is an Aztec diamond $A(n)$ of order $n$:
	\begin{equation*}
	\diam(\FF_{A(n)}) = \frac{n^{3}}{3} + \frac{n^{2}}{2} +  \frac{n}{6}
	\end{equation*}
	%\begin{equation*}
	%\diam(\FF_{S}) = \bigg{\lceil} \frac{N}{2} \bigg{\rceil} +
	%\sum_{i=1}^{N-1}{4 \bigg\lceil \frac{i}{2} \bigg\rceil}(N-i)
	%\end{equation*}
\end{thmC}

\input{aztecsquare}

\begin{proof}
We'll prove the formula for the rectangle and for the Aztec diamond (the square simply being the $R(n,n)$ rectangle). 

To begin, we note the self scaled-similarity the two figures have in common: by removing the ring $R_{1}$ of $R(m,n)$, resp. of $A(n)$, we obtain $R(m-2,n-2)$, resp. $A(n-2)$. We are also interested in the number of interior vertices (we denote by $\mathring{V}\left(X\right)$ the set of interior vertices of $X$).

For the rectangle, we have
\begin{align*}
|\{v\in \mathring{V}\left( R(m,n)\right)\}| & = (n-1)(m-1)
\end{align*}
We can also count those of the Aztec diamond column by column, starting from the central column:
\begin{align*}
|\{ v\in \mathring{V}\left(A(n)\right)\}|
	& = (2n-1) + 2\Big( (2n-3) + (2n-5) + \cdots + (2n-(2n-1)) \Big)\\
	& = 2\Big( \sum_{i=1}^{n}2n-(2i-1) \Big) - (2n-1)\\
	& = 2\Big( (2n+1)n - 2(n+1)\frac{n}{2} \Big) - (2n-1)\\
	& = 2n^{2} - 2n +1
\end{align*}
From the previous theorem we have
\begin{equation*}
\diam(\FF_{S})  = \sum_{v \in V(S)} \lev(v) = \sum_{i \geq 1}|\{v \in P: \lev(v) \geqslant i \}|
			%& = |\{v: lev(v) \geqslant 1 \}| + |\{v : lev(v) \geqslant 2 \}| \, + \cdots +
			%	|\{v : lev(v) \geqslant \bigg{\lceil}\frac{n}{2}\bigg{\rceil} \}|
\end{equation*}
For the rectangle this becomes
\begin{equation*}
\diam(\FF_{R(m,n)}) = \sum_{i=1}^{\lceil \frac{n}{2}\rceil}{\big(n-(2i-1)\big)\big(m-(2i-1)\big)}
\end{equation*}
and for the Aztec diamond
\begin{equation*}
\diam(\FF_{A(n)}) = \sum_{i=0}^{\lceil\frac{n}{2}\rceil - 1}{2(n-2i)^{2} - 2(n-2i) + 1}
\end{equation*}
The results follow by expanding the terms and by using the classical identities
\begin{equation*}
\sum_{i=1}^{m}{i}=(m+1)\frac{m}{2} \quad \quad \text{and} \quad \quad \sum_{i=1}^{m}{i^{2}}=(m+1)(2m+1)\frac{m}{6}
\end{equation*}
\end{proof}

\addcontentsline{toc}{section}{References}
\bibliographystyle{amsplain}

{\em Addresses:}

Department of Mathematics, University of Fribourg, Switzerland \\
{\em Email:} \href{mailto:hugo.parlier@unifr.ch}{hugo.parlier@unifr.ch}\\
{\em Email:} \href{mailto:samuel.zappa@unifr.ch}{samuel.zappa@unifr.ch}\\

\end{document}

%% file: intro.tex
\begin{figure}[h]
\begin{center}
\definecolor{ffffff}{rgb}{1.,1.,1.}
\definecolor{cqcqcq}{rgb}{0.752941176471,0.752941176471,0.752941176471}
\definecolor{aqaqaq}{rgb}{0.627450980392,0.627450980392,0.627450980392}
\subfloat[A tiling of a surface]{
\definecolor{cqcqcq}{rgb}{0.752941176471,0.752941176471,0.752941176471}
\definecolor{aqaqaq}{rgb}{0.627450980392,0.627450980392,0.627450980392}
\begin{tikzpicture}[line cap=round,line join=round,>=triangle 45,x=0.55cm,y=0.55cm]
\clip(0.9,0.9) rectangle (5.1,5.1);
\draw [color=aqaqaq] (1.,1.)-- (1.,5.);
\draw [line width=0.4pt,color=cqcqcq] (1.,2.)-- (1.,1.);
\draw [line width=0.4pt,color=cqcqcq] (1.,1.)-- (2.,1.);
\draw [line width=0.4pt,color=cqcqcq] (2.,1.)-- (2.,2.);
\draw [line width=0.4pt,color=cqcqcq] (2.,2.)-- (1.,2.);
\draw [line width=0.4pt,color=cqcqcq] (2.,3.)-- (2.,2.);
\draw [line width=0.4pt,color=cqcqcq] (2.,2.)-- (3.,2.);
\draw [line width=0.4pt,color=cqcqcq] (3.,2.)-- (3.,3.);
\draw [line width=0.4pt,color=cqcqcq] (3.,3.)-- (2.,3.);
\draw [line width=0.4pt,color=cqcqcq] (3.,4.)-- (3.,3.);
\draw [line width=0.4pt,color=cqcqcq] (3.,3.)-- (4.,3.);
\draw [line width=0.4pt,color=cqcqcq] (4.,3.)-- (4.,4.);
\draw [line width=0.4pt,color=cqcqcq] (4.,4.)-- (3.,4.);
\draw [line width=0.4pt,color=cqcqcq] (2.,5.)-- (2.,4.);
\draw [line width=0.4pt,color=cqcqcq] (2.,4.)-- (3.,4.);
\draw [line width=0.4pt,color=cqcqcq] (3.,4.)-- (3.,5.);
\draw [line width=0.4pt,color=cqcqcq] (3.,5.)-- (2.,5.);
\draw [line width=0.4pt,color=cqcqcq] (1.,4.)-- (1.,3.);
\draw [line width=0.4pt,color=cqcqcq] (1.,3.)-- (2.,3.);
\draw [line width=0.4pt,color=cqcqcq] (2.,3.)-- (2.,4.);
\draw [line width=0.4pt,color=cqcqcq] (2.,4.)-- (1.,4.);
\draw [line width=0.4pt,color=cqcqcq] (4.,3.)-- (4.,2.);
\draw [line width=0.4pt,color=cqcqcq] (4.,2.)-- (5.,2.);
\draw [line width=0.4pt,color=cqcqcq] (5.,2.)-- (5.,3.);
\draw [line width=0.4pt,color=cqcqcq] (5.,3.)-- (4.,3.);
\draw [color=cqcqcq] (1.,5.)-- (3.,5.);
\draw [color=cqcqcq] (3.,5.)-- (3.,4.);
\draw [color=cqcqcq] (3.,4.)-- (5.,4.);
\draw [color=cqcqcq] (5.,4.)-- (5.,2.);
\draw [color=cqcqcq] (5.,2.)-- (3.,2.);
\draw [color=cqcqcq] (3.,2.)-- (3.,1.);
\draw [color=cqcqcq] (3.,1.)-- (1.,1.);
\draw [line width=1.pt] (1.,1.)-- (3.,1.);
\draw [line width=1.pt] (3.,1.)-- (3.,2.);
\draw [line width=1.pt] (3.,2.)-- (1.,2.);
\draw [line width=1.pt] (1.,2.)-- (1.,1.);
\draw [line width=1.pt] (2.,2.)-- (2.,4.);
\draw [line width=1.pt] (3.,2.)-- (3.,4.);
\draw [line width=1.pt] (2.,4.)-- (3.,4.);
\draw [line width=1.pt] (1.,4.)-- (2.,4.);
\draw [line width=1.pt] (1.,4.)-- (1.,2.);
\draw [line width=1.pt] (1.,5.)-- (1.,4.);
\draw [line width=1.pt] (1.,5.)-- (3.,5.);
\draw [line width=1.pt] (3.,5.)-- (3.,4.);
\draw [line width=1.pt] (3.,4.)-- (4.,4.);
\draw [line width=1.pt] (4.,4.)-- (4.,2.);
\draw [line width=1.pt] (4.,2.)-- (3.,2.);
\draw [line width=1.pt] (4.,2.)-- (5.,2.);
\draw [line width=1.pt] (5.,2.)-- (5.,4.);
\draw [line width=1.pt] (5.,4.)-- (4.,4.);
\end{tikzpicture}} \quad \quad \quad
\subfloat[A new tiling obtained by a flip]{
\definecolor{cqcqcq}{rgb}{0.752941176471,0.752941176471,0.752941176471}
\definecolor{aqaqaq}{rgb}{0.627450980392,0.627450980392,0.627450980392}
\begin{tikzpicture}[line cap=round,line join=round,>=triangle 45,x=0.55cm,y=0.55cm]
\clip(0.9,0.9) rectangle (5.1,5.1);
\draw [color=aqaqaq] (1.,1.)-- (1.,5.);
\draw [line width=0.4pt,color=cqcqcq] (1.,2.)-- (1.,1.);
\draw [line width=0.4pt,color=cqcqcq] (1.,1.)-- (2.,1.);
\draw [line width=0.4pt,color=cqcqcq] (2.,1.)-- (2.,2.);
\draw [line width=0.4pt,color=cqcqcq] (2.,2.)-- (1.,2.);
\draw [line width=0.4pt,color=cqcqcq] (2.,3.)-- (2.,2.);
\draw [line width=0.4pt,color=cqcqcq] (2.,2.)-- (3.,2.);
\draw [line width=0.4pt,color=cqcqcq] (3.,2.)-- (3.,3.);
\draw [line width=0.4pt,color=cqcqcq] (3.,3.)-- (2.,3.);
\draw [line width=0.4pt,color=cqcqcq] (3.,4.)-- (3.,3.);
\draw [line width=0.4pt,color=cqcqcq] (3.,3.)-- (4.,3.);
\draw [line width=0.4pt,color=cqcqcq] (4.,3.)-- (4.,4.);
\draw [line width=0.4pt,color=cqcqcq] (4.,4.)-- (3.,4.);
\draw [line width=0.4pt,color=cqcqcq] (2.,5.)-- (2.,4.);
\draw [line width=0.4pt,color=cqcqcq] (2.,4.)-- (3.,4.);
\draw [line width=0.4pt,color=cqcqcq] (3.,4.)-- (3.,5.);
\draw [line width=0.4pt,color=cqcqcq] (3.,5.)-- (2.,5.);
\draw [line width=0.4pt,color=cqcqcq] (1.,4.)-- (1.,3.);
\draw [line width=0.4pt,color=cqcqcq] (1.,3.)-- (2.,3.);
\draw [line width=0.4pt,color=cqcqcq] (2.,3.)-- (2.,4.);
\draw [line width=0.4pt,color=cqcqcq] (2.,4.)-- (1.,4.);
\draw [line width=0.4pt,color=cqcqcq] (4.,3.)-- (4.,2.);
\draw [line width=0.4pt,color=cqcqcq] (4.,2.)-- (5.,2.);
\draw [line width=0.4pt,color=cqcqcq] (5.,2.)-- (5.,3.);
\draw [line width=0.4pt,color=cqcqcq] (5.,3.)-- (4.,3.);
\draw [color=cqcqcq] (1.,5.)-- (3.,5.);
\draw [color=cqcqcq] (3.,5.)-- (3.,4.);
\draw [color=cqcqcq] (3.,4.)-- (5.,4.);
\draw [color=cqcqcq] (5.,4.)-- (5.,2.);
\draw [color=cqcqcq] (5.,2.)-- (3.,2.);
\draw [color=cqcqcq] (3.,2.)-- (3.,1.);
\draw [color=cqcqcq] (3.,1.)-- (1.,1.);
\draw [line width=1pt] (1.,1.)-- (3.,1.);
\draw [line width=1pt] (3.,1.)-- (3.,2.);
\draw [line width=1pt] (3.,2.)-- (1.,2.);
\draw [line width=1pt] (1.,2.)-- (1.,1.);
\draw [line width=1pt] (2.,2.)-- (2.,4.);
\draw [line width=1pt] (2.,4.)-- (3.,4.);
\draw [line width=1.pt] (1.,4.)-- (2.,4.);
\draw [line width=1.pt] (1.,4.)-- (1.,2.);
\draw [line width=1.pt] (1.,5.)-- (1.,4.);
\draw [line width=1.pt] (1.,5.)-- (3.,5.);
\draw [line width=1.pt] (3.,5.)-- (3.,4.);
\draw [line width=1.pt] (3.,4.)-- (4.,4.);
\draw [line width=1.pt] (4.,4.)-- (4.,2.);
\draw [line width=1.pt] (4.,2.)-- (3.,2.);
\draw [line width=1.pt] (4.,2.)-- (5.,2.);
\draw [line width=1.pt] (5.,2.)-- (5.,4.);
\draw [line width=1.pt] (5.,4.)-- (4.,4.);
\draw [line width=1.pt] (2.,3.)-- (4.,3.);
\end{tikzpicture}} \quad \quad \quad
\subfloat[The dual representation]{
\definecolor{ffffff}{rgb}{1.,1.,1.}
\definecolor{cqcqcq}{rgb}{0.752941176471,0.752941176471,0.752941176471}
\definecolor{aqaqaq}{rgb}{0.627450980392,0.627450980392,0.627450980392}
\begin{tikzpicture}[line cap=round,line join=round,>=triangle 45,x=0.55cm,y=0.55cm]
\clip(0.9,0.9) rectangle (5.1,5.1);
\draw [color=aqaqaq] (1.,1.)-- (1.,5.);
\draw [line width=0.4pt,color=cqcqcq] (1.,2.)-- (1.,1.);
\draw [line width=0.4pt,color=cqcqcq] (1.,1.)-- (2.,1.);
\draw [line width=0.4pt,color=cqcqcq] (2.,1.)-- (2.,2.);
\draw [line width=0.4pt,color=cqcqcq] (2.,2.)-- (1.,2.);
\draw [line width=0.4pt,color=cqcqcq] (2.,3.)-- (2.,2.);
\draw [line width=0.4pt,color=cqcqcq] (2.,2.)-- (3.,2.);
\draw [line width=0.4pt,color=cqcqcq] (3.,2.)-- (3.,3.);
\draw [line width=0.4pt,color=cqcqcq] (3.,3.)-- (2.,3.);
\draw [line width=0.4pt,color=cqcqcq] (3.,4.)-- (3.,3.);
\draw [line width=0.4pt,color=cqcqcq] (3.,3.)-- (4.,3.);
\draw [line width=0.4pt,color=cqcqcq] (4.,3.)-- (4.,4.);
\draw [line width=0.4pt,color=cqcqcq] (4.,4.)-- (3.,4.);
\draw [line width=0.4pt,color=cqcqcq] (2.,5.)-- (2.,4.);
\draw [line width=0.4pt,color=cqcqcq] (2.,4.)-- (3.,4.);
\draw [line width=0.4pt,color=cqcqcq] (3.,4.)-- (3.,5.);
\draw [line width=0.4pt,color=cqcqcq] (3.,5.)-- (2.,5.);
\draw [line width=0.4pt,color=cqcqcq] (1.,4.)-- (1.,3.);
\draw [line width=0.4pt,color=cqcqcq] (1.,3.)-- (2.,3.);
\draw [line width=0.4pt,color=cqcqcq] (2.,3.)-- (2.,4.);
\draw [line width=0.4pt,color=cqcqcq] (2.,4.)-- (1.,4.);
\draw [line width=0.4pt,color=cqcqcq] (4.,3.)-- (4.,2.);
\draw [line width=0.4pt,color=cqcqcq] (4.,2.)-- (5.,2.);
\draw [line width=0.4pt,color=cqcqcq] (5.,2.)-- (5.,3.);
\draw [line width=0.4pt,color=cqcqcq] (5.,3.)-- (4.,3.);
\draw [color=cqcqcq] (1.,5.)-- (3.,5.);
\draw [color=cqcqcq] (3.,5.)-- (3.,4.);
\draw [color=cqcqcq] (3.,4.)-- (5.,4.);
\draw [color=cqcqcq] (5.,4.)-- (5.,2.);
\draw [color=cqcqcq] (5.,2.)-- (3.,2.);
\draw [color=cqcqcq] (3.,2.)-- (3.,1.);
\draw [color=cqcqcq] (3.,1.)-- (1.,1.);
\draw [line width=1pt](2.5,4.5)-- (1.5,4.5);
\draw [line width=1pt](1.5,3.5)-- (1.5,2.5);
\draw [line width=1pt](4.5,3.5)-- (4.5,2.5);
\draw [line width=1pt](1.5,1.5)-- (2.5,1.5);
\draw [line width=1pt](2.5,3.5)-- (3.5,3.5);
\draw [line width=1pt](2.5,2.5)-- (3.5,2.5);
\begin{scriptsize}
\draw [fill=ffffff] (1.5,1.5) circle (1.5pt);
\draw [fill=ffffff] (1.5,2.5) circle (1.5pt);
\draw [fill=ffffff] (2.5,2.5) circle (1.5pt);
\draw [fill=ffffff] (2.5,1.5) circle (1.5pt);
\draw [fill=ffffff] (2.5,4.5) circle (1.5pt);
\draw [fill=ffffff] (1.5,4.5) circle (1.5pt);
\draw [fill=ffffff] (1.5,3.5) circle (1.5pt);
\draw [fill=ffffff] (2.5,3.5) circle (1.5pt);
\draw [fill=ffffff] (3.5,2.5) circle (1.5pt);
\draw [fill=ffffff] (3.5,3.5) circle (1.5pt);
\draw [fill=ffffff] (4.5,2.5) circle (1.5pt);
\draw [fill=ffffff] (4.5,3.5) circle (1.5pt);
\end{scriptsize}
\end{tikzpicture}}

\caption{}
\label{ghx:intro}
\end{center}
\end{figure}

%% file: rectangle1.tex
\begin{figure}[h]
\begin{center}
\definecolor{ffffff}{rgb}{1.,1.,1.}
\definecolor{cqcqcq}{rgb}{0.752941176471,0.752941176471,0.752941176471}
\definecolor{aqaqaq}{rgb}{0.627450980392,0.627450980392,0.627450980392}
\subfloat[A tiling of a {\it n} x 2 rectangle]{
\begin{tikzpicture}[line cap=round,line join=round,>=triangle 45,x=0.60cm,y=0.60cm]
\clip(0.9,0.9) rectangle (11.1,3.1);
\draw [line width=0.4pt,color=cqcqcq] (1.,1.)-- (5.,1.);
\draw [line width=0.4pt,color=cqcqcq] (7.,1.)-- (11.,1.);
\draw [line width=0.4pt,color=cqcqcq] (11.,1.)-- (11.,3.);
\draw [line width=0.4pt,color=cqcqcq] (11.,3.)-- (7.,3.);
\draw [line width=0.4pt,color=cqcqcq] (5.,3.)-- (1.,3.);
\draw [line width=0.4pt,color=cqcqcq] (1.,3.)-- (1.,1.);
\draw [line width=0.4pt,color=cqcqcq] (1.,2.)-- (2.,2.);
\draw [line width=0.4pt,color=cqcqcq] (2.,2.)-- (2.,1.);
\draw [line width=0.4pt,color=cqcqcq] (2.,1.)-- (1.,1.);
\draw [line width=0.4pt,color=cqcqcq] (1.,1.)-- (1.,2.);
\draw [line width=0.4pt,color=cqcqcq] (2.,3.)-- (3.,3.);
\draw [line width=0.4pt,color=cqcqcq] (3.,3.)-- (3.,2.);
\draw [line width=0.4pt,color=cqcqcq] (3.,2.)-- (2.,2.);
\draw [line width=0.4pt,color=cqcqcq] (2.,2.)-- (2.,3.);
\draw [line width=0.4pt,color=cqcqcq] (3.,2.)-- (4.,2.);
\draw [line width=0.4pt,color=cqcqcq] (4.,2.)-- (4.,1.);
\draw [line width=0.4pt,color=cqcqcq] (4.,1.)-- (3.,1.);
\draw [line width=0.4pt,color=cqcqcq] (3.,1.)-- (3.,2.);
\draw [line width=0.4pt,color=cqcqcq] (4.,3.)-- (5.,3.);
\draw [line width=0.4pt,color=cqcqcq] (5.,3.)-- (5.,2.);
\draw [line width=0.4pt,color=cqcqcq] (5.,2.)-- (4.,2.);
\draw [line width=0.4pt,color=cqcqcq] (4.,2.)-- (4.,3.);
\draw [line width=0.4pt,color=cqcqcq] (7.,2.)-- (8.,2.);
\draw [line width=0.4pt,color=cqcqcq] (8.,2.)-- (8.,1.);
\draw [line width=0.4pt,color=cqcqcq] (8.,1.)-- (7.,1.);
\draw [line width=0.4pt,color=cqcqcq] (7.,1.)-- (7.,2.);
\draw [line width=0.4pt,color=cqcqcq] (8.,3.)-- (9.,3.);
\draw [line width=0.4pt,color=cqcqcq] (9.,3.)-- (9.,2.);
\draw [line width=0.4pt,color=cqcqcq] (9.,2.)-- (8.,2.);
\draw [line width=0.4pt,color=cqcqcq] (8.,2.)-- (8.,3.);
\draw [line width=0.4pt,color=cqcqcq] (9.,2.)-- (10.,2.);
\draw [line width=0.4pt,color=cqcqcq] (10.,2.)-- (10.,1.);
\draw [line width=0.4pt,color=cqcqcq] (10.,1.)-- (9.,1.);
\draw [line width=0.4pt,color=cqcqcq] (9.,1.)-- (9.,2.);
\draw [line width=0.4pt,color=cqcqcq] (10.,3.)-- (11.,3.);
\draw [line width=0.4pt,color=cqcqcq] (11.,3.)-- (11.,2.);
\draw [line width=0.4pt,color=cqcqcq] (11.,2.)-- (10.,2.);
\draw [line width=0.4pt,color=cqcqcq] (10.,2.)-- (10.,3.);
\draw [line width=0.4pt,dash pattern=on 3pt off 3pt,color=cqcqcq] (5.,3.)-- (7.,3.);
\draw [line width=0.4pt,dash pattern=on 3pt off 3pt,color=cqcqcq] (5.,1.)-- (7.,1.);
\draw [line width=0.4pt,color=cqcqcq] (5.,2.)-- (5.,1.);
\draw [line width=0.4pt,color=cqcqcq] (7.,2.)-- (7.,3.);
\draw [line width=1pt](1.5,2.5)-- (1.5,1.5);
\draw [line width=1pt](2.5,2.5)-- (3.5,2.5);
\draw [line width=1pt](2.5,1.5)-- (3.5,1.5);
\draw [line width=1pt](4.5,2.5)-- (5.,2.5);
\draw [line width=1pt](4.5,1.5)-- (5.,1.5);
\draw [line width=1pt,dotted] (5.,1.5)-- (5.3,1.5);
\draw [line width=1pt,dotted] (5.,2.5)-- (5.3,2.5);
\draw [line width=1pt](7.5,1.5)-- (7.5,2.5);
\draw [line width=1pt](8.5,2.5)-- (8.5,1.5);
\draw [line width=1pt](9.5,2.5)-- (10.5,2.5);
\draw [line width=1pt](10.5,1.5)-- (9.5,1.5);
\begin{scriptsize}
\draw [fill=ffffff] (1.5,2.5) circle (1.5pt);
\draw [fill=ffffff] (1.5,1.5) circle (1.5pt);
\draw [fill=ffffff] (2.5,1.5) circle (1.5pt);
\draw [fill=ffffff] (2.5,2.5) circle (1.5pt);
\draw [fill=ffffff] (3.5,2.5) circle (1.5pt);
\draw [fill=ffffff] (3.5,1.5) circle (1.5pt);
\draw [fill=ffffff] (4.5,1.5) circle (1.5pt);
\draw [fill=ffffff] (4.5,2.5) circle (1.5pt);
\draw [fill=ffffff] (7.5,2.5) circle (1.5pt);
\draw [fill=ffffff] (7.5,1.5) circle (1.5pt);
\draw [fill=ffffff] (8.5,1.5) circle (1.5pt);
\draw [fill=ffffff] (8.5,2.5) circle (1.5pt);
\draw [fill=ffffff] (9.5,2.5) circle (1.5pt);
\draw [fill=ffffff] (9.5,1.5) circle (1.5pt);
\draw [fill=ffffff] (10.5,2.5) circle (1.5pt);
\draw [fill=ffffff] (10.5,1.5) circle (1.5pt);
\end{scriptsize}
\end{tikzpicture}} \quad \quad 
\subfloat[Our standard tiling]{
\begin{tikzpicture}[line cap=round,line join=round,>=triangle 45,x=0.60cm,y=0.60cm]
\clip(0.9,0.9) rectangle (11.1,3.1);
\draw [line width=0.4pt,color=cqcqcq] (1.,1.)-- (5.,1.);
\draw [line width=0.4pt,color=cqcqcq] (7.,1.)-- (11.,1.);
\draw [line width=0.4pt,color=cqcqcq] (11.,1.)-- (11.,3.);
\draw [line width=0.4pt,color=cqcqcq] (11.,3.)-- (7.,3.);
\draw [line width=0.4pt,color=cqcqcq] (5.,3.)-- (1.,3.);
\draw [line width=0.4pt,color=cqcqcq] (1.,3.)-- (1.,1.);
\draw [line width=0.4pt,color=cqcqcq] (1.,2.)-- (2.,2.);
\draw [line width=0.4pt,color=cqcqcq] (2.,2.)-- (2.,1.);
\draw [line width=0.4pt,color=cqcqcq] (2.,1.)-- (1.,1.);
\draw [line width=0.4pt,color=cqcqcq] (1.,1.)-- (1.,2.);
\draw [line width=0.4pt,color=cqcqcq] (2.,3.)-- (3.,3.);
\draw [line width=0.4pt,color=cqcqcq] (3.,3.)-- (3.,2.);
\draw [line width=0.4pt,color=cqcqcq] (3.,2.)-- (2.,2.);
\draw [line width=0.4pt,color=cqcqcq] (2.,2.)-- (2.,3.);
\draw [line width=0.4pt,color=cqcqcq] (3.,2.)-- (4.,2.);
\draw [line width=0.4pt,color=cqcqcq] (4.,2.)-- (4.,1.);
\draw [line width=0.4pt,color=cqcqcq] (4.,1.)-- (3.,1.);
\draw [line width=0.4pt,color=cqcqcq] (3.,1.)-- (3.,2.);
\draw [line width=0.4pt,color=cqcqcq] (4.,3.)-- (5.,3.);
\draw [line width=0.4pt,color=cqcqcq] (5.,3.)-- (5.,2.);
\draw [line width=0.4pt,color=cqcqcq] (5.,2.)-- (4.,2.);
\draw [line width=0.4pt,color=cqcqcq] (4.,2.)-- (4.,3.);
\draw [line width=0.4pt,color=cqcqcq] (7.,2.)-- (8.,2.);
\draw [line width=0.4pt,color=cqcqcq] (8.,2.)-- (8.,1.);
\draw [line width=0.4pt,color=cqcqcq] (8.,1.)-- (7.,1.);
\draw [line width=0.4pt,color=cqcqcq] (7.,1.)-- (7.,2.);
\draw [line width=0.4pt,color=cqcqcq] (8.,3.)-- (9.,3.);
\draw [line width=0.4pt,color=cqcqcq] (9.,3.)-- (9.,2.);
\draw [line width=0.4pt,color=cqcqcq] (9.,2.)-- (8.,2.);
\draw [line width=0.4pt,color=cqcqcq] (8.,2.)-- (8.,3.);
\draw [line width=0.4pt,color=cqcqcq] (9.,2.)-- (10.,2.);
\draw [line width=0.4pt,color=cqcqcq] (10.,2.)-- (10.,1.);
\draw [line width=0.4pt,color=cqcqcq] (10.,1.)-- (9.,1.);
\draw [line width=0.4pt,color=cqcqcq] (9.,1.)-- (9.,2.);
\draw [line width=0.4pt,color=cqcqcq] (10.,3.)-- (11.,3.);
\draw [line width=0.4pt,color=cqcqcq] (11.,3.)-- (11.,2.);
\draw [line width=0.4pt,color=cqcqcq] (11.,2.)-- (10.,2.);
\draw [line width=0.4pt,color=cqcqcq] (10.,2.)-- (10.,3.);
\draw [line width=0.4pt,dash pattern=on 3pt off 3pt,color=cqcqcq] (5.,3.)-- (7.,3.);
\draw [line width=0.4pt,dash pattern=on 3pt off 3pt,color=cqcqcq] (5.,1.)-- (7.,1.);
\draw [line width=0.4pt,color=cqcqcq] (5.,2.)-- (5.,1.);
\draw [line width=0.4pt,color=cqcqcq] (7.,2.)-- (7.,3.);
\draw [line width=1pt](1.5,2.5)-- (1.5,1.5);
\draw [line width=1pt](2.5,2.5)-- (2.5,1.5);
\draw [line width=1pt](3.5,2.5)-- (3.5,1.5);
\draw [line width=1pt](4.5,2.5)-- (4.5,1.5);
\draw [line width=1pt](7.5,2.5)-- (7.5,1.5);
\draw [line width=1pt](8.5,2.5)-- (8.5,1.5);
\draw [line width=1pt](9.5,2.5)-- (9.5,1.5);
\draw [line width=1pt](10.5,2.5)-- (10.5,1.5);
\begin{scriptsize}
\draw [fill=ffffff] (1.5,2.5) circle (1.5pt);
\draw [fill=ffffff] (1.5,1.5) circle (1.5pt);
\draw [fill=ffffff] (2.5,1.5) circle (1.5pt);
\draw [fill=ffffff] (2.5,2.5) circle (1.5pt);
\draw [fill=ffffff] (3.5,2.5) circle (1.5pt);
\draw [fill=ffffff] (3.5,1.5) circle (1.5pt);
\draw [fill=ffffff] (4.5,1.5) circle (1.5pt);
\draw [fill=ffffff] (4.5,2.5) circle (1.5pt);
\draw [fill=ffffff] (7.5,2.5) circle (1.5pt);
\draw [fill=ffffff] (7.5,1.5) circle (1.5pt);
\draw [fill=ffffff] (8.5,1.5) circle (1.5pt);
\draw [fill=ffffff] (8.5,2.5) circle (1.5pt);
\draw [fill=ffffff] (9.5,2.5) circle (1.5pt);
\draw [fill=ffffff] (9.5,1.5) circle (1.5pt);
\draw [fill=ffffff] (10.5,2.5) circle (1.5pt);
\draw [fill=ffffff] (10.5,1.5) circle (1.5pt);
\end{scriptsize}
\end{tikzpicture}}

\caption{}
\label{ghx:rectangle1}
\end{center}
\end{figure}

%% file: rectangle2.tex
\begin{figure}[h]
\begin{center}
\definecolor{ffffff}{rgb}{1.,1.,1.}
\definecolor{cqcqcq}{rgb}{0.752941176471,0.752941176471,0.752941176471}
\definecolor{aqaqaq}{rgb}{0.627450980392,0.627450980392,0.627450980392}
\subfloat{
\begin{tikzpicture}[line cap=round,line join=round,>=triangle 45,x=0.60cm,y=0.60cm]
\clip(0.9,0.9) rectangle (11.1,3.1);
\draw [line width=0.4pt,color=cqcqcq] (1.,1.)-- (5.,1.);
\draw [line width=0.4pt,color=cqcqcq] (7.,1.)-- (11.,1.);
\draw [line width=0.4pt,color=cqcqcq] (11.,1.)-- (11.,3.);
\draw [line width=0.4pt,color=cqcqcq] (11.,3.)-- (7.,3.);
\draw [line width=0.4pt,color=cqcqcq] (5.,3.)-- (1.,3.);
\draw [line width=0.4pt,color=cqcqcq] (1.,3.)-- (1.,1.);
\draw [line width=0.4pt,color=cqcqcq] (1.,2.)-- (2.,2.);
\draw [line width=0.4pt,color=cqcqcq] (2.,2.)-- (2.,1.);
\draw [line width=0.4pt,color=cqcqcq] (2.,1.)-- (1.,1.);
\draw [line width=0.4pt,color=cqcqcq] (1.,1.)-- (1.,2.);
\draw [line width=0.4pt,color=cqcqcq] (2.,3.)-- (3.,3.);
\draw [line width=0.4pt,color=cqcqcq] (3.,3.)-- (3.,2.);
\draw [line width=0.4pt,color=cqcqcq] (3.,2.)-- (2.,2.);
\draw [line width=0.4pt,color=cqcqcq] (2.,2.)-- (2.,3.);
\draw [line width=0.4pt,color=cqcqcq] (3.,2.)-- (4.,2.);
\draw [line width=0.4pt,color=cqcqcq] (4.,2.)-- (4.,1.);
\draw [line width=0.4pt,color=cqcqcq] (4.,1.)-- (3.,1.);
\draw [line width=0.4pt,color=cqcqcq] (3.,1.)-- (3.,2.);
\draw [line width=0.4pt,color=cqcqcq] (4.,3.)-- (5.,3.);
\draw [line width=0.4pt,color=cqcqcq] (5.,3.)-- (5.,2.);
\draw [line width=0.4pt,color=cqcqcq] (5.,2.)-- (4.,2.);
\draw [line width=0.4pt,color=cqcqcq] (4.,2.)-- (4.,3.);
\draw [line width=0.4pt,color=cqcqcq] (7.,2.)-- (8.,2.);
\draw [line width=0.4pt,color=cqcqcq] (8.,2.)-- (8.,1.);
\draw [line width=0.4pt,color=cqcqcq] (8.,1.)-- (7.,1.);
\draw [line width=0.4pt,color=cqcqcq] (7.,1.)-- (7.,2.);
\draw [line width=0.4pt,color=cqcqcq] (8.,3.)-- (9.,3.);
\draw [line width=0.4pt,color=cqcqcq] (9.,3.)-- (9.,2.);
\draw [line width=0.4pt,color=cqcqcq] (9.,2.)-- (8.,2.);
\draw [line width=0.4pt,color=cqcqcq] (8.,2.)-- (8.,3.);
\draw [line width=0.4pt,color=cqcqcq] (9.,2.)-- (10.,2.);
\draw [line width=0.4pt,color=cqcqcq] (10.,2.)-- (10.,1.);
\draw [line width=0.4pt,color=cqcqcq] (10.,1.)-- (9.,1.);
\draw [line width=0.4pt,color=cqcqcq] (9.,1.)-- (9.,2.);
\draw [line width=0.4pt,color=cqcqcq] (10.,3.)-- (11.,3.);
\draw [line width=0.4pt,color=cqcqcq] (11.,3.)-- (11.,2.);
\draw [line width=0.4pt,color=cqcqcq] (11.,2.)-- (10.,2.);
\draw [line width=0.4pt,color=cqcqcq] (10.,2.)-- (10.,3.);
\draw [line width=0.4pt,dash pattern=on 3pt off 3pt,color=cqcqcq] (5.,3.)-- (7.,3.);
\draw [line width=0.4pt,dash pattern=on 3pt off 3pt,color=cqcqcq] (5.,1.)-- (7.,1.);
\draw [line width=0.4pt,color=cqcqcq] (5.,2.)-- (5.,1.);
\draw [line width=0.4pt,color=cqcqcq] (7.,2.)-- (7.,3.);
\draw [line width=1pt](1.5,2.5)-- (2.5,2.5);
\draw [line width=1pt](1.5,1.5)-- (2.5,1.5);
\draw [line width=1pt](3.5,2.5)-- (4.5,2.5);
\draw [line width=1pt](3.5,1.5)-- (4.5,1.5);
\draw [line width=1pt](7.5,2.5)-- (8.5,2.5);
\draw [line width=1pt](7.5,1.5)-- (8.5,1.5);
\draw [line width=1pt](9.5,2.5)-- (10.5,2.5);
\draw [line width=1pt](9.5,1.5)-- (10.5,1.5);
\begin{scriptsize}
\draw [fill=ffffff] (1.5,2.5) circle (1.5pt);
\draw [fill=ffffff] (1.5,1.5) circle (1.5pt);
\draw [fill=ffffff] (2.5,1.5) circle (1.5pt);
\draw [fill=ffffff] (2.5,2.5) circle (1.5pt);
\draw [fill=ffffff] (3.5,2.5) circle (1.5pt);
\draw [fill=ffffff] (3.5,1.5) circle (1.5pt);
\draw [fill=ffffff] (4.5,1.5) circle (1.5pt);
\draw [fill=ffffff] (4.5,2.5) circle (1.5pt);
\draw [fill=ffffff] (7.5,2.5) circle (1.5pt);
\draw [fill=ffffff] (7.5,1.5) circle (1.5pt);
\draw [fill=ffffff] (8.5,1.5) circle (1.5pt);
\draw [fill=ffffff] (8.5,2.5) circle (1.5pt);
\draw [fill=ffffff] (9.5,2.5) circle (1.5pt);
\draw [fill=ffffff] (9.5,1.5) circle (1.5pt);
\draw [fill=ffffff] (10.5,2.5) circle (1.5pt);
\draw [fill=ffffff] (10.5,1.5) circle (1.5pt);
\end{scriptsize}
\end{tikzpicture}} \quad \quad \quad
\subfloat{
\begin{tikzpicture}[line cap=round,line join=round,>=triangle 45,x=0.60cm,y=0.60cm]
\clip(0.9,0.9) rectangle (11.1,3.1);
\draw [line width=0.4pt,color=cqcqcq] (1.,1.)-- (5.,1.);
\draw [line width=0.4pt,color=cqcqcq] (7.,1.)-- (11.,1.);
\draw [line width=0.4pt,color=cqcqcq] (11.,1.)-- (11.,3.);
\draw [line width=0.4pt,color=cqcqcq] (11.,3.)-- (7.,3.);
\draw [line width=0.4pt,color=cqcqcq] (5.,3.)-- (1.,3.);
\draw [line width=0.4pt,color=cqcqcq] (1.,3.)-- (1.,1.);
\draw [line width=0.4pt,color=cqcqcq] (1.,2.)-- (2.,2.);
\draw [line width=0.4pt,color=cqcqcq] (2.,2.)-- (2.,1.);
\draw [line width=0.4pt,color=cqcqcq] (2.,1.)-- (1.,1.);
\draw [line width=0.4pt,color=cqcqcq] (1.,1.)-- (1.,2.);
\draw [line width=0.4pt,color=cqcqcq] (2.,3.)-- (3.,3.);
\draw [line width=0.4pt,color=cqcqcq] (3.,3.)-- (3.,2.);
\draw [line width=0.4pt,color=cqcqcq] (3.,2.)-- (2.,2.);
\draw [line width=0.4pt,color=cqcqcq] (2.,2.)-- (2.,3.);
\draw [line width=0.4pt,color=cqcqcq] (3.,2.)-- (4.,2.);
\draw [line width=0.4pt,color=cqcqcq] (4.,2.)-- (4.,1.);
\draw [line width=0.4pt,color=cqcqcq] (4.,1.)-- (3.,1.);
\draw [line width=0.4pt,color=cqcqcq] (3.,1.)-- (3.,2.);
\draw [line width=0.4pt,color=cqcqcq] (4.,3.)-- (5.,3.);
\draw [line width=0.4pt,color=cqcqcq] (5.,3.)-- (5.,2.);
\draw [line width=0.4pt,color=cqcqcq] (5.,2.)-- (4.,2.);
\draw [line width=0.4pt,color=cqcqcq] (4.,2.)-- (4.,3.);
\draw [line width=0.4pt,color=cqcqcq] (7.,2.)-- (8.,2.);
\draw [line width=0.4pt,color=cqcqcq] (8.,2.)-- (8.,1.);
\draw [line width=0.4pt,color=cqcqcq] (8.,1.)-- (7.,1.);
\draw [line width=0.4pt,color=cqcqcq] (7.,1.)-- (7.,2.);
\draw [line width=0.4pt,color=cqcqcq] (8.,3.)-- (9.,3.);
\draw [line width=0.4pt,color=cqcqcq] (9.,3.)-- (9.,2.);
\draw [line width=0.4pt,color=cqcqcq] (9.,2.)-- (8.,2.);
\draw [line width=0.4pt,color=cqcqcq] (8.,2.)-- (8.,3.);
\draw [line width=0.4pt,color=cqcqcq] (9.,2.)-- (10.,2.);
\draw [line width=0.4pt,color=cqcqcq] (10.,2.)-- (10.,1.);
\draw [line width=0.4pt,color=cqcqcq] (10.,1.)-- (9.,1.);
\draw [line width=0.4pt,color=cqcqcq] (9.,1.)-- (9.,2.);
\draw [line width=0.4pt,color=cqcqcq] (10.,3.)-- (11.,3.);
\draw [line width=0.4pt,color=cqcqcq] (11.,3.)-- (11.,2.);
\draw [line width=0.4pt,color=cqcqcq] (11.,2.)-- (10.,2.);
\draw [line width=0.4pt,color=cqcqcq] (10.,2.)-- (10.,3.);
\draw [line width=0.4pt,dash pattern=on 3pt off 3pt,color=cqcqcq] (5.,3.)-- (7.,3.);
\draw [line width=0.4pt,dash pattern=on 3pt off 3pt,color=cqcqcq] (5.,1.)-- (7.,1.);
\draw [line width=0.4pt,color=cqcqcq] (5.,2.)-- (5.,1.);
\draw [line width=0.4pt,color=cqcqcq] (7.,2.)-- (7.,3.);
\draw [line width=1pt](1.5,2.5)-- (1.5,1.5);
\draw [line width=1pt](2.5,2.5)-- (3.5,2.5);
\draw [line width=1pt](2.5,1.5)-- (3.5,1.5);
\draw [line width=1pt](4.5,2.5)-- (5.,2.5);
\draw [line width=1pt](4.5,1.5)-- (5.,1.5);
\draw [line width=1pt,dotted] (5.,1.5)-- (5.3,1.5);
\draw [line width=1pt,dotted] (5.,2.5)-- (5.3,2.5);
\draw [line width=1pt](10.5,2.5)-- (10.5,1.5);
\draw [line width=1pt](9.5,2.5)-- (8.5,2.5);
\draw [line width=1pt](8.5,1.5)-- (9.5,1.5);
\draw [line width=1pt](7.5,2.5)-- (7.,2.5);
\draw [line width=1pt](7.5,1.5)-- (7.,1.5);
\draw [line width=1pt,dotted] (7.,2.5)-- (6.7,2.5);
\draw [line width=1pt,dotted] (7.,1.5)-- (6.7,1.5);
\begin{scriptsize}
\draw [fill=ffffff] (1.5,2.5) circle (1.5pt);
\draw [fill=ffffff] (1.5,1.5) circle (1.5pt);
\draw [fill=ffffff] (2.5,1.5) circle (1.5pt);
\draw [fill=ffffff] (2.5,2.5) circle (1.5pt);
\draw [fill=ffffff] (3.5,2.5) circle (1.5pt);
\draw [fill=ffffff] (3.5,1.5) circle (1.5pt);
\draw [fill=ffffff] (4.5,1.5) circle (1.5pt);
\draw [fill=ffffff] (4.5,2.5) circle (1.5pt);
\draw [fill=ffffff] (7.5,2.5) circle (1.5pt);
\draw [fill=ffffff] (7.5,1.5) circle (1.5pt);
\draw [fill=ffffff] (8.5,1.5) circle (1.5pt);
\draw [fill=ffffff] (8.5,2.5) circle (1.5pt);
\draw [fill=ffffff] (9.5,2.5) circle (1.5pt);
\draw [fill=ffffff] (9.5,1.5) circle (1.5pt);
\draw [fill=ffffff] (10.5,2.5) circle (1.5pt);
\draw [fill=ffffff] (10.5,1.5) circle (1.5pt);
\end{scriptsize}
\end{tikzpicture}}

\caption{}
\label{ghx:rectangle}
\end{center}
\end{figure}

%% file: surface.tex
\begin{figure}
\begin{center}
\begin{tikzpicture}[line cap=round,line join=round,>=triangle 45,x=0.5cm,y=0.5cm]
\clip(0.9,0.9) rectangle (5.1,6.1);
\draw (1.,5.)-- (1.,3.)-- (4.,3.)-- (4.,2.)-- (3.,2.)-- (3.,3.)-- (2.,3.)-- (2.,1.)-- (5.,1.)-- (5.,5.);
\draw (2.,5.)-- (2.,6.)-- (3.,6.)-- (3.,5.)-- (5.,5.);
\draw (1.,5.)-- (3.,5.);
\draw [dash pattern=on 2pt off 2pt] (3.,3.)-- (3.,5.);
\draw [dash pattern=on 2pt off 2pt] (2.,3.)-- (2.,5.);
\end{tikzpicture}
\caption{A surface $S$ which cannot be embedded into $\R^{2}$}
\label{ghx:surface}
\end{center}
\end{figure}
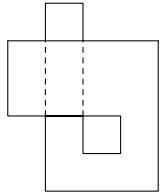

%% file: height.tex
\begin{figure}[h]
\centering
\definecolor{ffffff}{rgb}{1.,1.,1.}
\definecolor{urwzer}{rgb}{0.254901960784,0.411764705882,0.882352941176}
\definecolor{aqaqaq}{rgb}{0.627450980392,0.627450980392,0.627450980392}
\definecolor{yqqqqq}{rgb}{0.501960784314,0.,0.}
\definecolor{cqcqcq}{rgb}{0.752941176471,0.752941176471,0.752941176471}
\subfloat[Different paths give $v$ the same height]{
\begin{tikzpicture}[line cap=round,line join=round,>=triangle 45,x=0.7cm,y=0.7cm]
\clip(-0.7,-0.7) rectangle (5.7,4.7);
\fill[line width=0.pt,color=cqcqcq,fill=cqcqcq,fill opacity=0.65] (0.,0.) -- (1.,0.) -- (1.,1.) -- (0.,1.) -- cycle;
\fill[line width=0.pt,color=cqcqcq,fill=cqcqcq,fill opacity=0.65] (1.,2.) -- (1.,1.) -- (2.,1.) -- (2.,2.) -- cycle;
\fill[line width=0.pt,color=cqcqcq,fill=cqcqcq,fill opacity=0.65] (1.,3.) -- (2.,3.) -- (2.,4.) -- (1.,4.) -- cycle;
\fill[line width=0.pt,color=cqcqcq,fill=cqcqcq,fill opacity=0.65] (2.,3.) -- (3.,3.) -- (3.,2.) -- (2.,2.) -- cycle;
\fill[line width=0.pt,color=cqcqcq,fill=cqcqcq,fill opacity=0.65] (3.,3.) -- (3.,4.) -- (4.,4.) -- (4.,3.) -- cycle;
\fill[line width=0.pt,color=cqcqcq,fill=cqcqcq,fill opacity=0.65] (4.,3.) -- (4.,2.) -- (5.,2.) -- (5.,3.) -- cycle;
\fill[line width=0.pt,color=cqcqcq,fill=cqcqcq,fill opacity=0.65] (3.,2.) -- (4.,2.) -- (4.,1.) -- (3.,1.) -- cycle;
\fill[line width=0.pt,color=cqcqcq,fill=cqcqcq,fill opacity=0.65] (2.,1.) -- (3.,1.) -- (3.,0.) -- (2.,0.) -- cycle;
\draw [color=cqcqcq] (0.,0.)-- (0.,2.);
\draw [color=cqcqcq] (0.,2.)-- (1.,2.);
\draw [color=cqcqcq] (1.,2.)-- (1.,4.);
\draw [color=cqcqcq] (1.,4.)-- (4.,4.);
\draw [color=cqcqcq] (4.,4.)-- (4.,3.);
\draw [color=cqcqcq] (4.,3.)-- (5.,3.);
\draw [color=cqcqcq] (5.,3.)-- (5.,1.);
\draw [color=cqcqcq] (5.,1.)-- (4.,1.);
\draw [color=cqcqcq] (4.,1.)-- (4.,0.);
\draw [color=cqcqcq] (4.,0.)-- (0.,0.);
\draw [color=cqcqcq] (0.,0.)-- (1.,0.);
\draw [color=cqcqcq] (1.,0.)-- (1.,1.);
\draw [color=cqcqcq] (1.,1.)-- (0.,1.);
\draw [color=cqcqcq] (0.,1.)-- (0.,0.);
\draw [color=cqcqcq] (1.,2.)-- (1.,1.);
\draw [color=cqcqcq] (1.,1.)-- (2.,1.);
\draw [color=cqcqcq] (2.,1.)-- (2.,2.);
\draw [color=cqcqcq] (2.,2.)-- (1.,2.);
\draw [color=cqcqcq] (1.,3.)-- (2.,3.);
\draw [color=cqcqcq] (2.,3.)-- (2.,4.);
\draw [color=cqcqcq] (2.,4.)-- (1.,4.);
\draw [color=cqcqcq] (1.,4.)-- (1.,3.);
\draw [color=cqcqcq] (2.,3.)-- (3.,3.);
\draw [color=cqcqcq] (3.,3.)-- (3.,2.);
\draw [color=cqcqcq] (3.,2.)-- (2.,2.);
\draw [color=cqcqcq] (2.,2.)-- (2.,3.);
\draw [color=cqcqcq] (3.,3.)-- (3.,4.);
\draw [color=cqcqcq] (3.,4.)-- (4.,4.);
\draw [color=cqcqcq] (4.,4.)-- (4.,3.);
\draw [color=cqcqcq] (4.,3.)-- (3.,3.);
\draw [color=cqcqcq] (4.,3.)-- (4.,2.);
\draw [color=cqcqcq] (4.,2.)-- (5.,2.);
\draw [color=cqcqcq] (5.,2.)-- (5.,3.);
\draw [color=cqcqcq] (5.,3.)-- (4.,3.);
\draw [color=cqcqcq] (3.,2.)-- (4.,2.);
\draw [color=cqcqcq] (4.,2.)-- (4.,1.);
\draw [color=cqcqcq] (4.,1.)-- (3.,1.);
\draw [color=cqcqcq] (3.,1.)-- (3.,2.);
\draw [color=cqcqcq] (2.,1.)-- (3.,1.);
\draw [color=cqcqcq] (3.,1.)-- (3.,0.);
\draw [color=cqcqcq] (3.,0.)-- (2.,0.);
\draw [color=cqcqcq] (2.,0.)-- (2.,1.);
\draw [color=aqaqaq] (0.,1.5)-- (-0.2,1.7);
\draw [color=aqaqaq] (0.,1.5)-- (0.2,1.7);
\draw [color=aqaqaq] (4.,2.5)-- (3.8,2.3);
\draw [color=aqaqaq] (4.,2.5)-- (4.2,2.3);
\draw [color=aqaqaq] (2.,2.5)-- (1.8,2.3);
\draw [color=aqaqaq] (2.,2.5)-- (2.2,2.3);
\draw [color=aqaqaq] (3.,3.5)-- (2.8,3.3);
\draw [color=aqaqaq] (3.,3.5)-- (3.2,3.3);
\draw [color=aqaqaq] (1.,3.5)-- (0.8,3.3);
\draw [color=aqaqaq] (1.,3.5)-- (1.2,3.3);
\draw [color=aqaqaq] (1.,1.5)-- (0.8,1.3);
\draw [color=aqaqaq] (1.,1.5)-- (1.2,1.3);
\draw [color=aqaqaq] (3.,1.5)-- (2.8,1.3);
\draw [color=aqaqaq] (3.,1.5)-- (3.2,1.3);
\draw [color=aqaqaq] (5.,1.5)-- (4.8,1.3);
\draw [color=aqaqaq] (5.,1.5)-- (5.2,1.3);
\draw [color=aqaqaq] (2.,0.5)-- (1.8,0.3);
\draw [color=aqaqaq] (2.,0.5)-- (2.2,0.3);
\draw [color=aqaqaq] (0.7,2.2)-- (0.5,2.);
\draw [color=aqaqaq] (0.5,2.)-- (0.7,1.8);
\draw [color=aqaqaq] (0.,0.5)-- (-0.2,0.3);
\draw [color=aqaqaq] (0.,0.5)-- (0.2,0.3);
\draw [color=aqaqaq] (4.,0.5)-- (3.8,0.3);
\draw [color=aqaqaq] (4.,0.5)-- (4.2,0.3);
\draw [color=aqaqaq] (4.3,3.2)-- (4.5,3.);
\draw [color=aqaqaq] (4.5,3.)-- (4.3,2.8);
\draw [color=aqaqaq] (3.3,4.2)-- (3.5,4.);
\draw [color=aqaqaq] (3.5,4.)-- (3.3,3.8);
\draw [color=aqaqaq] (1.3,4.2)-- (1.5,4.);
\draw [color=aqaqaq] (1.5,4.)-- (1.3,3.8);
\draw [color=aqaqaq] (2.3,3.2)-- (2.5,3.);
\draw [color=aqaqaq] (2.5,3.)-- (2.3,2.8);
\draw [color=aqaqaq] (3.3,2.2)-- (3.5,2.);
\draw [color=aqaqaq] (3.5,2.)-- (3.3,1.8);
\draw [color=aqaqaq] (1.3,2.2)-- (1.5,2.);
\draw [color=aqaqaq] (1.5,2.)-- (1.3,1.8);
\draw [color=aqaqaq] (0.3,1.2)-- (0.5,1.);
\draw [color=aqaqaq] (0.5,1.)-- (0.3,0.8);
\draw [color=aqaqaq] (2.3,1.2)-- (2.5,1.);
\draw [color=aqaqaq] (2.5,1.)-- (2.3,0.8);
\draw [color=aqaqaq] (4.3,1.2)-- (4.5,1.);
\draw [color=aqaqaq] (4.5,1.)-- (4.3,0.8);
\draw [color=aqaqaq] (3.7,3.2)-- (3.5,3.);
\draw [color=aqaqaq] (3.5,3.)-- (3.7,2.8);
\draw [color=aqaqaq] (1.7,3.2)-- (1.5,3.);
\draw [color=aqaqaq] (1.5,3.)-- (1.7,2.8);
\draw [color=aqaqaq] (2.7,2.2)-- (2.5,2.);
\draw [color=aqaqaq] (2.5,2.)-- (2.7,1.8);
\draw [color=aqaqaq] (4.7,2.2)-- (4.5,2.);
\draw [color=aqaqaq] (4.5,2.)-- (4.7,1.8);
\draw [color=aqaqaq] (3.7,1.2)-- (3.5,1.);
\draw [color=aqaqaq] (3.5,1.)-- (3.7,0.8);
\draw [color=aqaqaq] (1.7,1.2)-- (1.5,1.);
\draw [color=aqaqaq] (1.5,1.)-- (1.7,0.8);
\draw [color=aqaqaq] (0.7,0.2)-- (0.5,0.);
\draw [color=aqaqaq] (0.5,0.)-- (0.7,-0.2);
\draw [color=aqaqaq] (2.7,0.2)-- (2.5,0.);
\draw [color=aqaqaq] (2.5,0.)-- (2.7,-0.2);
\draw [color=aqaqaq] (1.3,0.2)-- (1.5,0.);
\draw [color=aqaqaq] (1.5,0.)-- (1.3,-0.2);
\draw [color=aqaqaq] (3.3,0.2)-- (3.5,0.);
\draw [color=aqaqaq] (3.5,0.)-- (3.3,-0.2);
\draw [color=aqaqaq] (2.7,4.2)-- (2.5,4.);
\draw [color=aqaqaq] (2.5,4.)-- (2.7,3.8);
\draw [color=aqaqaq] (1.8,3.7)-- (2.,3.5);
\draw [color=aqaqaq] (2.,3.5)-- (2.2,3.7);
\draw [color=aqaqaq] (3.8,3.7)-- (4.,3.5);
\draw [color=aqaqaq] (4.,3.5)-- (4.2,3.7);
\draw [color=aqaqaq] (4.8,2.7)-- (5.,2.5);
\draw [color=aqaqaq] (5.,2.5)-- (5.2,2.7);
\draw [color=aqaqaq] (3.8,1.7)-- (4.,1.5);
\draw [color=aqaqaq] (4.,1.5)-- (4.2,1.7);
\draw [color=aqaqaq] (1.8,1.7)-- (2.,1.5);
\draw [color=aqaqaq] (2.,1.5)-- (2.2,1.7);
\draw [color=aqaqaq] (0.8,0.7)-- (1.,0.5);
\draw [color=aqaqaq] (1.,0.5)-- (1.2,0.7);
\draw [color=aqaqaq] (2.8,0.7)-- (3.,0.5);
\draw [color=aqaqaq] (3.,0.5)-- (3.2,0.7);
\draw [color=aqaqaq] (0.8,2.7)-- (1.,2.5);
\draw [color=aqaqaq] (1.,2.5)-- (1.2,2.7);
\draw [color=aqaqaq] (2.8,2.7)-- (3.,2.5);
\draw [color=aqaqaq] (3.,2.5)-- (3.2,2.7);
\draw (4.,3.)-- (4.,1.);
\draw (4.,1.)-- (5.,1.);
\draw (5.,1.)-- (5.,3.);
\draw (5.,3.)-- (4.,3.);
\draw (3.,4.)-- (3.,2.);
\draw (3.,2.)-- (4.,2.);
\draw (4.,2.)-- (4.,4.);
\draw (4.,4.)-- (3.,4.);
\draw (2.,1.)-- (2.,0.);
\draw (2.,0.)-- (4.,0.);
\draw (4.,0.)-- (4.,1.);
\draw (4.,1.)-- (2.,1.);
\draw (2.,2.)-- (2.,1.);
\draw (2.,1.)-- (4.,1.);
\draw (4.,1.)-- (4.,2.);
\draw (4.,2.)-- (2.,2.);
\draw (0.,2.)-- (0.,1.);
\draw (0.,1.)-- (2.,1.);
\draw (2.,1.)-- (2.,2.);
\draw (2.,2.)-- (0.,2.);
\draw (0.,1.)-- (0.,0.);
\draw (0.,0.)-- (2.,0.);
\draw (2.,0.)-- (2.,1.);
\draw (2.,1.)-- (0.,1.);
\draw (1.,3.)-- (1.,2.);
\draw (1.,2.)-- (3.,2.);
\draw (3.,2.)-- (3.,3.);
\draw (3.,3.)-- (1.,3.);
\draw (1.,4.)-- (1.,3.);
\draw (1.,3.)-- (3.,3.);
\draw (3.,3.)-- (3.,4.);
\draw (3.,4.)-- (1.,4.);
\draw [line width=1.2pt] (0.,0.)-- (4.,0.);
\draw [line width=1.2pt] (4.,2.)-- (4.,0.);
\draw [line width=1.2pt] (4.,2.)-- (2.,2.);
\draw [line width=1.2pt] (2.,1.)-- (2.,2.);

\draw [line width=1.2pt] (0.8,0.)-- (0.6,0.2);
\draw [line width=1.2pt] (0.8,0.)-- (0.6,-0.2);

\draw [line width=1.2pt,dotted] (0.,0.)-- (0.,1.);
\draw [line width=1.2pt,dotted] (2.,1.)-- (0.,1.);

\draw [line width=1.2pt,dotted] (0.,0.8)-- (0.2,0.6);
\draw [line width=1.2pt,dotted] (0.,0.8)-- (-0.2,0.6);

\draw (1.45,1.05) node[anchor=north west] {\footnotesize$v$};
\draw (-0.5,0) node[anchor=north west] {\footnotesize$v_{0}$};
\draw (-0.1,0.6) node[anchor=north west] {\footnotesize$0$};
\draw (0.9,0.6) node[anchor=north west] {\footnotesize-$1$};
\draw (1.9,0.6) node[anchor=north west] {\footnotesize$0$};
\draw (2.9,0.6) node[anchor=north west] {\footnotesize-$1$};
\draw (3.9,0.6) node[anchor=north west] {\footnotesize$0$};
\draw (-0.1,1.6) node[anchor=north west] {\footnotesize$1$};
\draw (0.9,1.6) node[anchor=north west] {\footnotesize$2$};
\draw (1.9,1.6) node[anchor=north west] {\footnotesize$1$};
\draw (3.9,1.6) node[anchor=north west] {\footnotesize$1$};
\draw (1.9,2.6) node[anchor=north west] {\footnotesize$0$};
\draw (2.9,2.6) node[anchor=north west] {\footnotesize-$1$};
\draw (3.9,2.6) node[anchor=north west] {\footnotesize$0$};
\draw [fill] (0,0) circle (1.5pt);
\draw [fill] (2.,1.) circle (1.5pt);
\end{tikzpicture}} \quad \quad \quad \quad
\subfloat[Set of heights of a tiling]{
\begin{tikzpicture}[line cap=round,line join=round,>=triangle 45,x=0.7cm,y=0.7cm]
\clip(-0.7,-0.7) rectangle (5.7,4.7);
\fill[line width=0.pt,color=cqcqcq,fill=cqcqcq,fill opacity=0.65] (0.,0.) -- (1.,0.) -- (1.,1.) -- (0.,1.) -- cycle;
\fill[line width=0.pt,color=cqcqcq,fill=cqcqcq,fill opacity=0.65] (1.,2.) -- (1.,1.) -- (2.,1.) -- (2.,2.) -- cycle;
\fill[line width=0.pt,color=cqcqcq,fill=cqcqcq,fill opacity=0.65] (1.,3.) -- (2.,3.) -- (2.,4.) -- (1.,4.) -- cycle;
\fill[line width=0.pt,color=cqcqcq,fill=cqcqcq,fill opacity=0.65] (2.,3.) -- (3.,3.) -- (3.,2.) -- (2.,2.) -- cycle;
\fill[line width=0.pt,color=cqcqcq,fill=cqcqcq,fill opacity=0.65] (3.,3.) -- (3.,4.) -- (4.,4.) -- (4.,3.) -- cycle;
\fill[line width=0.pt,color=cqcqcq,fill=cqcqcq,fill opacity=0.65] (4.,3.) -- (4.,2.) -- (5.,2.) -- (5.,3.) -- cycle;
\fill[line width=0.pt,color=cqcqcq,fill=cqcqcq,fill opacity=0.65] (3.,2.) -- (4.,2.) -- (4.,1.) -- (3.,1.) -- cycle;
\fill[line width=0.pt,color=cqcqcq,fill=cqcqcq,fill opacity=0.65] (2.,1.) -- (3.,1.) -- (3.,0.) -- (2.,0.) -- cycle;
\draw [color=cqcqcq] (0.,0.)-- (0.,2.);
\draw [color=cqcqcq] (0.,2.)-- (1.,2.);
\draw [color=cqcqcq] (1.,2.)-- (1.,4.);
\draw [color=cqcqcq] (1.,4.)-- (4.,4.);
\draw [color=cqcqcq] (4.,4.)-- (4.,3.);
\draw [color=cqcqcq] (4.,3.)-- (5.,3.);
\draw [color=cqcqcq] (5.,3.)-- (5.,1.);
\draw [color=cqcqcq] (5.,1.)-- (4.,1.);
\draw [color=cqcqcq] (4.,1.)-- (4.,0.);
\draw [color=cqcqcq] (4.,0.)-- (0.,0.);
\draw [color=cqcqcq] (0.,0.)-- (1.,0.);
\draw [color=cqcqcq] (1.,0.)-- (1.,1.);
\draw [color=cqcqcq] (1.,1.)-- (0.,1.);
\draw [color=cqcqcq] (0.,1.)-- (0.,0.);
\draw [color=cqcqcq] (1.,2.)-- (1.,1.);
\draw [color=cqcqcq] (1.,1.)-- (2.,1.);
\draw [color=cqcqcq] (2.,1.)-- (2.,2.);
\draw [color=cqcqcq] (2.,2.)-- (1.,2.);
\draw [color=cqcqcq] (1.,3.)-- (2.,3.);
\draw [color=cqcqcq] (2.,3.)-- (2.,4.);
\draw [color=cqcqcq] (2.,4.)-- (1.,4.);
\draw [color=cqcqcq] (1.,4.)-- (1.,3.);
\draw [color=cqcqcq] (2.,3.)-- (3.,3.);
\draw [color=cqcqcq] (3.,3.)-- (3.,2.);
\draw [color=cqcqcq] (3.,2.)-- (2.,2.);
\draw [color=cqcqcq] (2.,2.)-- (2.,3.);
\draw [color=cqcqcq] (3.,3.)-- (3.,4.);
\draw [color=cqcqcq] (3.,4.)-- (4.,4.);
\draw [color=cqcqcq] (4.,4.)-- (4.,3.);
\draw [color=cqcqcq] (4.,3.)-- (3.,3.);
\draw [color=cqcqcq] (4.,3.)-- (4.,2.);
\draw [color=cqcqcq] (4.,2.)-- (5.,2.);
\draw [color=cqcqcq] (5.,2.)-- (5.,3.);
\draw [color=cqcqcq] (5.,3.)-- (4.,3.);
\draw [color=cqcqcq] (3.,2.)-- (4.,2.);
\draw [color=cqcqcq] (4.,2.)-- (4.,1.);
\draw [color=cqcqcq] (4.,1.)-- (3.,1.);
\draw [color=cqcqcq] (3.,1.)-- (3.,2.);
\draw [color=cqcqcq] (2.,1.)-- (3.,1.);
\draw [color=cqcqcq] (3.,1.)-- (3.,0.);
\draw [color=cqcqcq] (3.,0.)-- (2.,0.);
\draw [color=cqcqcq] (2.,0.)-- (2.,1.);
\draw [color=aqaqaq] (0.,1.5)-- (-0.2,1.7);
\draw [color=aqaqaq] (0.,1.5)-- (0.2,1.7);
\draw [color=aqaqaq] (4.,2.5)-- (3.8,2.3);
\draw [color=aqaqaq] (4.,2.5)-- (4.2,2.3);
\draw [color=aqaqaq] (2.,2.5)-- (1.8,2.3);
\draw [color=aqaqaq] (2.,2.5)-- (2.2,2.3);
\draw [color=aqaqaq] (3.,3.5)-- (2.8,3.3);
\draw [color=aqaqaq] (3.,3.5)-- (3.2,3.3);
\draw [color=aqaqaq] (1.,3.5)-- (0.8,3.3);
\draw [color=aqaqaq] (1.,3.5)-- (1.2,3.3);
\draw [color=aqaqaq] (1.,1.5)-- (0.8,1.3);
\draw [color=aqaqaq] (1.,1.5)-- (1.2,1.3);
\draw [color=aqaqaq] (3.,1.5)-- (2.8,1.3);
\draw [color=aqaqaq] (3.,1.5)-- (3.2,1.3);
\draw [color=aqaqaq] (5.,1.5)-- (4.8,1.3);
\draw [color=aqaqaq] (5.,1.5)-- (5.2,1.3);
\draw [color=aqaqaq] (2.,0.5)-- (1.8,0.3);
\draw [color=aqaqaq] (2.,0.5)-- (2.2,0.3);
\draw [color=aqaqaq] (0.7,2.2)-- (0.5,2.);
\draw [color=aqaqaq] (0.5,2.)-- (0.7,1.8);
\draw [color=aqaqaq] (0.,0.5)-- (-0.2,0.3);
\draw [color=aqaqaq] (0.,0.5)-- (0.2,0.3);
\draw [color=aqaqaq] (4.,0.5)-- (3.8,0.3);
\draw [color=aqaqaq] (4.,0.5)-- (4.2,0.3);
\draw [color=aqaqaq] (4.3,3.2)-- (4.5,3.);
\draw [color=aqaqaq] (4.5,3.)-- (4.3,2.8);
\draw [color=aqaqaq] (3.3,4.2)-- (3.5,4.);
\draw [color=aqaqaq] (3.5,4.)-- (3.3,3.8);
\draw [color=aqaqaq] (1.3,4.2)-- (1.5,4.);
\draw [color=aqaqaq] (1.5,4.)-- (1.3,3.8);
\draw [color=aqaqaq] (2.3,3.2)-- (2.5,3.);
\draw [color=aqaqaq] (2.5,3.)-- (2.3,2.8);
\draw [color=aqaqaq] (3.3,2.2)-- (3.5,2.);
\draw [color=aqaqaq] (3.5,2.)-- (3.3,1.8);
\draw [color=aqaqaq] (1.3,2.2)-- (1.5,2.);
\draw [color=aqaqaq] (1.5,2.)-- (1.3,1.8);
\draw [color=aqaqaq] (0.3,1.2)-- (0.5,1.);
\draw [color=aqaqaq] (0.5,1.)-- (0.3,0.8);
\draw [color=aqaqaq] (2.3,1.2)-- (2.5,1.);
\draw [color=aqaqaq] (2.5,1.)-- (2.3,0.8);
\draw [color=aqaqaq] (4.3,1.2)-- (4.5,1.);
\draw [color=aqaqaq] (4.5,1.)-- (4.3,0.8);
\draw [color=aqaqaq] (3.7,3.2)-- (3.5,3.);
\draw [color=aqaqaq] (3.5,3.)-- (3.7,2.8);
\draw [color=aqaqaq] (1.7,3.2)-- (1.5,3.);
\draw [color=aqaqaq] (1.5,3.)-- (1.7,2.8);
\draw [color=aqaqaq] (2.7,2.2)-- (2.5,2.);
\draw [color=aqaqaq] (2.5,2.)-- (2.7,1.8);
\draw [color=aqaqaq] (4.7,2.2)-- (4.5,2.);
\draw [color=aqaqaq] (4.5,2.)-- (4.7,1.8);
\draw [color=aqaqaq] (3.7,1.2)-- (3.5,1.);
\draw [color=aqaqaq] (3.5,1.)-- (3.7,0.8);
\draw [color=aqaqaq] (1.7,1.2)-- (1.5,1.);
\draw [color=aqaqaq] (1.5,1.)-- (1.7,0.8);
\draw [color=aqaqaq] (0.7,0.2)-- (0.5,0.);
\draw [color=aqaqaq] (0.5,0.)-- (0.7,-0.2);
\draw [color=aqaqaq] (2.7,0.2)-- (2.5,0.);
\draw [color=aqaqaq] (2.5,0.)-- (2.7,-0.2);
\draw [color=aqaqaq] (1.3,0.2)-- (1.5,0.);
\draw [color=aqaqaq] (1.5,0.)-- (1.3,-0.2);
\draw [color=aqaqaq] (3.3,0.2)-- (3.5,0.);
\draw [color=aqaqaq] (3.5,0.)-- (3.3,-0.2);
\draw [color=aqaqaq] (2.7,4.2)-- (2.5,4.);
\draw [color=aqaqaq] (2.5,4.)-- (2.7,3.8);
\draw [color=aqaqaq] (1.8,3.7)-- (2.,3.5);
\draw [color=aqaqaq] (2.,3.5)-- (2.2,3.7);
\draw [color=aqaqaq] (3.8,3.7)-- (4.,3.5);
\draw [color=aqaqaq] (4.,3.5)-- (4.2,3.7);
\draw [color=aqaqaq] (4.8,2.7)-- (5.,2.5);
\draw [color=aqaqaq] (5.,2.5)-- (5.2,2.7);
\draw [color=aqaqaq] (3.8,1.7)-- (4.,1.5);
\draw [color=aqaqaq] (4.,1.5)-- (4.2,1.7);
\draw [color=aqaqaq] (1.8,1.7)-- (2.,1.5);
\draw [color=aqaqaq] (2.,1.5)-- (2.2,1.7);
\draw [color=aqaqaq] (0.8,0.7)-- (1.,0.5);
\draw [color=aqaqaq] (1.,0.5)-- (1.2,0.7);
\draw [color=aqaqaq] (2.8,0.7)-- (3.,0.5);
\draw [color=aqaqaq] (3.,0.5)-- (3.2,0.7);
\draw [color=aqaqaq] (0.8,2.7)-- (1.,2.5);
\draw [color=aqaqaq] (1.,2.5)-- (1.2,2.7);
\draw [color=aqaqaq] (2.8,2.7)-- (3.,2.5);
\draw [color=aqaqaq] (3.,2.5)-- (3.2,2.7);
\draw (4.,3.)-- (4.,1.);
\draw (4.,1.)-- (5.,1.);
\draw (5.,1.)-- (5.,3.);
\draw (5.,3.)-- (4.,3.);
\draw (3.,4.)-- (3.,2.);
\draw (3.,2.)-- (4.,2.);
\draw (4.,2.)-- (4.,4.);
\draw (4.,4.)-- (3.,4.);
\draw (2.,1.)-- (2.,0.);
\draw (2.,0.)-- (4.,0.);
\draw (4.,0.)-- (4.,1.);
\draw (4.,1.)-- (2.,1.);
\draw (2.,2.)-- (2.,1.);
\draw (2.,1.)-- (4.,1.);
\draw (4.,1.)-- (4.,2.);
\draw (4.,2.)-- (2.,2.);
\draw (0.,2.)-- (0.,1.);
\draw (0.,1.)-- (2.,1.);
\draw (2.,1.)-- (2.,2.);
\draw (2.,2.)-- (0.,2.);
\draw (0.,1.)-- (0.,0.);
\draw (0.,0.)-- (2.,0.);
\draw (2.,0.)-- (2.,1.);
\draw (2.,1.)-- (0.,1.);
\draw (1.,3.)-- (1.,2.);
\draw (1.,2.)-- (3.,2.);
\draw (3.,2.)-- (3.,3.);
\draw (3.,3.)-- (1.,3.);
\draw (1.,4.)-- (1.,3.);
\draw (1.,3.)-- (3.,3.);
\draw (3.,3.)-- (3.,4.);
\draw (3.,4.)-- (1.,4.);
\draw (-0.5,0) node[anchor=north west] {\footnotesize$v_{0}$};
\draw (-0.1,0.6) node[anchor=north west] {\footnotesize$0$};
\draw (0.9,0.6) node[anchor=north west] {\footnotesize-$1$};
\draw (1.9,0.6) node[anchor=north west] {\footnotesize$0$};
\draw (2.9,0.6) node[anchor=north west] {\footnotesize-$1$};
\draw (3.9,0.6) node[anchor=north west] {\footnotesize$0$};
\draw (-0.1,1.6) node[anchor=north west] {\footnotesize$1$};
\draw (0.9,1.6) node[anchor=north west] {\footnotesize$2$};
\draw (1.9,1.6) node[anchor=north west] {\footnotesize$1$};
\draw (2.9,1.6) node[anchor=north west] {\footnotesize$2$};
\draw (3.9,1.6) node[anchor=north west] {\footnotesize$1$};
\draw (4.9,1.6) node[anchor=north west] {\footnotesize$2$};
\draw (-0.1,2.6) node[anchor=north west] {\footnotesize$0$};
\draw (0.9,2.6) node[anchor=north west] {\footnotesize-$1$};
\draw (1.9,2.6) node[anchor=north west] {\footnotesize$0$};
\draw (2.9,2.6) node[anchor=north west] {\footnotesize-$1$};
\draw (3.9,2.6) node[anchor=north west] {\footnotesize$0$};
\draw (4.9,2.6) node[anchor=north west] {\footnotesize$3$};
\draw (0.9,3.6) node[anchor=north west] {\footnotesize-$2$};
\draw (1.9,3.6) node[anchor=north west] {\footnotesize-$3$};
\draw (2.9,3.6) node[anchor=north west] {\footnotesize-$2$};
\draw (3.9,3.6) node[anchor=north west] {\footnotesize$1$};
\draw (4.9,3.6) node[anchor=north west] {\footnotesize$2$};
\draw (0.9,4.6) node[anchor=north west] {\footnotesize-$1$};
\draw (1.9,4.6) node[anchor=north west] {\footnotesize$0$};
\draw (2.9,4.6) node[anchor=north west] {\footnotesize-$1$};
\draw (3.9,4.6) node[anchor=north west] {\footnotesize$0$};
\draw [fill] (0,0) circle (1.5pt);
\end{tikzpicture}}
\caption{}
\label{ghx:path}
\end{figure}

%% file: distance.tex
\begin{figure}[h]
\begin{center}
\definecolor{ffffff}{rgb}{1.,1.,1.}
\definecolor{cqcqcq}{rgb}{0.752941176471,0.752941176471,0.752941176471}
\definecolor{aqaqaq}{rgb}{0.627450980392,0.627450980392,0.627450980392}
\subfloat[$\CC(T,T')$ obtained by superimposing the tilings $T$ and $T'$ (dotted)]{
\begin{tikzpicture}[line cap=round,line join=round,>=triangle 45,x=0.65cm,y=0.65cm]
\clip(0.9,0.9) rectangle (8.1,5.1);
\fill[line width=0.4pt,color=cqcqcq,fill=cqcqcq,fill opacity=0.5] (1.,2.) -- (1.,1.) -- (2.,1.) -- (2.,2.) -- cycle;
\fill[line width=0.4pt,color=cqcqcq,fill=cqcqcq,fill opacity=0.5] (2.,3.) -- (2.,2.) -- (3.,2.) -- (3.,3.) -- cycle;
\fill[line width=0.4pt,color=cqcqcq,fill=cqcqcq,fill opacity=0.5] (3.,4.) -- (3.,3.) -- (4.,3.) -- (4.,4.) -- cycle;
\fill[line width=0.4pt,color=cqcqcq,fill=cqcqcq,fill opacity=0.5] (4.,5.) -- (4.,4.) -- (5.,4.) -- (5.,5.) -- cycle;
\fill[line width=0.4pt,color=cqcqcq,fill=cqcqcq,fill opacity=0.5] (2.,5.) -- (2.,4.) -- (3.,4.) -- (3.,5.) -- cycle;
\fill[line width=0.4pt,color=cqcqcq,fill=cqcqcq,fill opacity=0.5] (1.,4.) -- (1.,3.) -- (2.,3.) -- (2.,4.) -- cycle;
\fill[line width=0.4pt,color=cqcqcq,fill=cqcqcq,fill opacity=0.5] (3.,2.) -- (3.,1.) -- (4.,1.) -- (4.,2.) -- cycle;
\fill[line width=0.4pt,color=cqcqcq,fill=cqcqcq,fill opacity=0.5] (4.,3.) -- (4.,2.) -- (5.,2.) -- (5.,3.) -- cycle;
\fill[line width=0.4pt,color=cqcqcq,fill=cqcqcq,fill opacity=0.5] (5.,4.) -- (5.,3.) -- (6.,3.) -- (6.,4.) -- cycle;
\fill[line width=0.4pt,color=cqcqcq,fill=cqcqcq,fill opacity=0.5] (6.,5.) -- (6.,4.) -- (7.,4.) -- (7.,5.) -- cycle;
\fill[line width=0.4pt,color=cqcqcq,fill=cqcqcq,fill opacity=0.5] (7.,4.) -- (7.,3.) -- (8.,3.) -- (8.,4.) -- cycle;
\fill[line width=0.4pt,color=cqcqcq,fill=cqcqcq,fill opacity=0.5] (6.,3.) -- (6.,2.) -- (7.,2.) -- (7.,3.) -- cycle;
\fill[line width=0.4pt,color=cqcqcq,fill=cqcqcq,fill opacity=0.5] (7.,2.) -- (7.,1.) -- (8.,1.) -- (8.,2.) -- cycle;
\fill[line width=0.4pt,color=cqcqcq,fill=cqcqcq,fill opacity=0.5] (5.,2.) -- (5.,1.) -- (6.,1.) -- (6.,2.) -- cycle;
\draw [color=aqaqaq] (1.,1.)-- (1.,5.);
\draw [color=aqaqaq] (1.,5.)-- (8.,5.);
\draw [color=aqaqaq] (8.,5.)-- (8.,1.);
\draw [color=aqaqaq] (8.,1.)-- (1.,1.);
\draw [line width=0.4pt,color=cqcqcq] (1.,2.)-- (1.,1.);
\draw [line width=0.4pt,color=cqcqcq] (1.,1.)-- (2.,1.);
\draw [line width=0.4pt,color=cqcqcq] (2.,1.)-- (2.,2.);
\draw [line width=0.4pt,color=cqcqcq] (2.,2.)-- (1.,2.);
\draw [line width=0.4pt,color=cqcqcq] (2.,3.)-- (2.,2.);
\draw [line width=0.4pt,color=cqcqcq] (2.,2.)-- (3.,2.);
\draw [line width=0.4pt,color=cqcqcq] (3.,2.)-- (3.,3.);
\draw [line width=0.4pt,color=cqcqcq] (3.,3.)-- (2.,3.);
\draw [line width=0.4pt,color=cqcqcq] (3.,4.)-- (3.,3.);
\draw [line width=0.4pt,color=cqcqcq] (3.,3.)-- (4.,3.);
\draw [line width=0.4pt,color=cqcqcq] (4.,3.)-- (4.,4.);
\draw [line width=0.4pt,color=cqcqcq] (4.,4.)-- (3.,4.);
\draw [line width=0.4pt,color=cqcqcq] (4.,5.)-- (4.,4.);
\draw [line width=0.4pt,color=cqcqcq] (4.,4.)-- (5.,4.);
\draw [line width=0.4pt,color=cqcqcq] (5.,4.)-- (5.,5.);
\draw [line width=0.4pt,color=cqcqcq] (5.,5.)-- (4.,5.);
\draw [line width=0.4pt,color=cqcqcq] (2.,5.)-- (2.,4.);
\draw [line width=0.4pt,color=cqcqcq] (2.,4.)-- (3.,4.);
\draw [line width=0.4pt,color=cqcqcq] (3.,4.)-- (3.,5.);
\draw [line width=0.4pt,color=cqcqcq] (3.,5.)-- (2.,5.);
\draw [line width=0.4pt,color=cqcqcq] (1.,4.)-- (1.,3.);
\draw [line width=0.4pt,color=cqcqcq] (1.,3.)-- (2.,3.);
\draw [line width=0.4pt,color=cqcqcq] (2.,3.)-- (2.,4.);
\draw [line width=0.4pt,color=cqcqcq] (2.,4.)-- (1.,4.);
\draw [line width=0.4pt,color=cqcqcq] (3.,2.)-- (3.,1.);
\draw [line width=0.4pt,color=cqcqcq] (3.,1.)-- (4.,1.);
\draw [line width=0.4pt,color=cqcqcq] (4.,1.)-- (4.,2.);
\draw [line width=0.4pt,color=cqcqcq] (4.,2.)-- (3.,2.);
\draw [line width=0.4pt,color=cqcqcq] (4.,3.)-- (4.,2.);
\draw [line width=0.4pt,color=cqcqcq] (4.,2.)-- (5.,2.);
\draw [line width=0.4pt,color=cqcqcq] (5.,2.)-- (5.,3.);
\draw [line width=0.4pt,color=cqcqcq] (5.,3.)-- (4.,3.);
\draw [line width=0.4pt,color=cqcqcq] (5.,4.)-- (5.,3.);
\draw [line width=0.4pt,color=cqcqcq] (5.,3.)-- (6.,3.);
\draw [line width=0.4pt,color=cqcqcq] (6.,3.)-- (6.,4.);
\draw [line width=0.4pt,color=cqcqcq] (6.,4.)-- (5.,4.);
\draw [line width=0.4pt,color=cqcqcq] (6.,5.)-- (6.,4.);
\draw [line width=0.4pt,color=cqcqcq] (6.,4.)-- (7.,4.);
\draw [line width=0.4pt,color=cqcqcq] (7.,4.)-- (7.,5.);
\draw [line width=0.4pt,color=cqcqcq] (7.,5.)-- (6.,5.);
\draw [line width=0.4pt,color=cqcqcq] (7.,4.)-- (7.,3.);
\draw [line width=0.4pt,color=cqcqcq] (7.,3.)-- (8.,3.);
\draw [line width=0.4pt,color=cqcqcq] (8.,3.)-- (8.,4.);
\draw [line width=0.4pt,color=cqcqcq] (8.,4.)-- (7.,4.);
\draw [line width=0.4pt,color=cqcqcq] (6.,3.)-- (6.,2.);
\draw [line width=0.4pt,color=cqcqcq] (6.,2.)-- (7.,2.);
\draw [line width=0.4pt,color=cqcqcq] (7.,2.)-- (7.,3.);
\draw [line width=0.4pt,color=cqcqcq] (7.,3.)-- (6.,3.);
\draw [line width=0.4pt,color=cqcqcq] (7.,2.)-- (7.,1.);
\draw [line width=0.4pt,color=cqcqcq] (7.,1.)-- (8.,1.);
\draw [line width=0.4pt,color=cqcqcq] (8.,1.)-- (8.,2.);
\draw [line width=0.4pt,color=cqcqcq] (8.,2.)-- (7.,2.);
\draw [line width=0.4pt,color=cqcqcq] (5.,2.)-- (5.,1.);
\draw [line width=0.4pt,color=cqcqcq] (5.,1.)-- (6.,1.);
\draw [line width=0.4pt,color=cqcqcq] (6.,1.)-- (6.,2.);
\draw [line width=0.4pt,color=cqcqcq] (6.,2.)-- (5.,2.);
\draw (2.5,2.5)-- (2.5,1.5);
\draw (3.5,1.5)-- (4.5,1.5);
\draw (5.5,1.5)-- (6.5,1.5);
\draw (7.5,1.5)-- (7.5,2.5);
\draw (7.5,3.5)-- (7.5,4.5);
\draw (6.5,4.5)-- (5.5,4.5);
\draw (4.5,4.5)-- (3.5,4.5);
\draw (2.5,4.5)-- (1.5,4.5);
\draw (1.5,3.5)-- (2.5,3.5);
\draw (3.5,2.5)-- (3.5,3.5);
\draw (4.5,2.5)-- (4.5,3.5);
\draw (5.5,2.5)-- (6.5,2.5);
\draw (5.5,3.5)-- (6.5,3.5);
\draw [line width=0.6pt,dotted] (1.5,4.5)-- (1.5,3.5);
\draw [line width=0.6pt,dotted] (2.5,3.5)-- (2.5,2.5);
\draw [line width=0.6pt,dotted] (2.5,1.5)-- (3.5,1.5);
\draw [line width=0.6pt,dotted] (4.5,1.5)-- (5.5,1.5);
\draw [line width=0.6pt,dotted] (6.5,1.5)-- (7.5,1.5);
\draw [line width=0.6pt,dotted] (7.5,2.5)-- (7.5,3.5);
\draw [line width=0.6pt,dotted] (7.5,4.5)-- (6.5,4.5);
\draw [line width=0.6pt,dotted] (5.5,4.5)-- (4.5,4.5);
\draw [line width=0.6pt,dotted] (3.5,4.5)-- (2.5,4.5);
\draw [line width=0.6pt,dotted] (3.5,3.5)-- (4.5,3.5);
\draw [line width=0.6pt,dotted] (3.5,2.5)-- (4.5,2.5);
\draw [line width=0.6pt,dotted] (5.5,3.5)-- (5.5,2.5);
\draw [line width=0.6pt,dotted] (6.5,3.5)-- (6.5,2.5);
\draw [line width=0.6pt,dotted] (1.51,2.5)-- (1.51,1.5);
\draw (4.5,3.)-- (4.3,2.8);
\draw (4.5,3.)-- (4.7,2.8);
\draw (6.,2.5)-- (6.2,2.7);
\draw (6.,2.5)-- (6.2,2.3);
\draw (7.5,2.)-- (7.3,1.8);
\draw (7.5,2.)-- (7.7,1.8);
\draw (1.49,2.5)-- (1.49,1.5);
\begin{scriptsize}
\draw [fill=cqcqcq] (1.5,1.5) circle (1.5pt);
\draw [fill=ffffff] (1.5,2.5) circle (1.5pt);
\draw [fill=cqcqcq] (2.5,2.5) circle (1.5pt);
\draw [fill=ffffff] (2.5,1.5) circle (1.5pt);
\draw [fill=cqcqcq] (3.5,1.5) circle (1.5pt);
\draw [fill=ffffff] (4.5,1.5) circle (1.5pt);
\draw [fill=cqcqcq] (5.5,1.5) circle (1.5pt);
\draw [fill=ffffff] (6.5,1.5) circle (1.5pt);
\draw [fill=cqcqcq] (7.5,1.5) circle (1.5pt);
\draw [fill=ffffff] (7.5,2.5) circle (1.5pt);
\draw [fill=cqcqcq] (7.5,3.5) circle (1.5pt);
\draw [fill=ffffff] (7.5,4.5) circle (1.5pt);
\draw [fill=cqcqcq] (6.5,4.5) circle (1.5pt);
\draw [fill=ffffff] (5.5,4.5) circle (1.5pt);
\draw [fill=cqcqcq] (4.5,4.5) circle (1.5pt);
\draw [fill=ffffff] (3.5,4.5) circle (1.5pt);
\draw [fill=cqcqcq] (2.5,4.5) circle (1.5pt);
\draw [fill=ffffff] (1.5,4.5) circle (1.5pt);
\draw [fill=cqcqcq] (1.5,3.5) circle (1.5pt);
\draw [fill=ffffff] (2.5,3.5) circle (1.5pt);
\draw [fill=ffffff] (3.5,2.5) circle (1.5pt);
\draw [fill=cqcqcq] (3.5,3.5) circle (1.5pt);
\draw [fill=cqcqcq] (4.5,2.5) circle (1.5pt);
\draw [fill=ffffff] (4.5,3.5) circle (1.5pt);
\draw [fill=ffffff] (5.5,2.5) circle (1.5pt);
\draw [fill=cqcqcq] (6.5,2.5) circle (1.5pt);
\draw [fill=cqcqcq] (5.5,3.5) circle (1.5pt);
\draw [fill=ffffff] (6.5,3.5) circle (1.5pt);
\end{scriptsize}
\end{tikzpicture}} \quad \quad
\subfloat[The associated filling shape showing $d_{\FF_S}(T,T') = 16$]{
\includegraphics[scale=0.24]{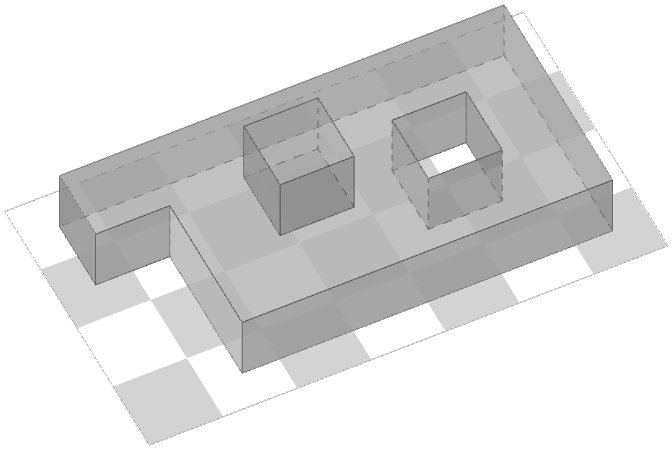}}
\caption{}
\label{ghx:distance}
\end{center}
\end{figure}

%% file: diameter.tex
\begin{figure}[h]
\begin{center}
\includegraphics[scale=0.24]{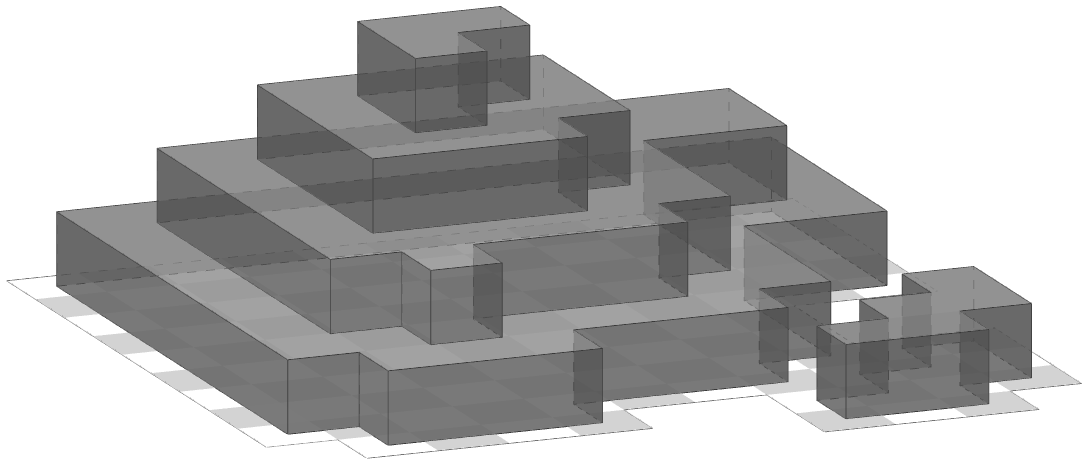}
\caption*{\footnotesize A filling shape of volume of $137$ realizes the diameter of the flip graph of \autoref{ghx:distance2}}
\caption{}
\label{ghx:diameter}
\end{center}
\end{figure}

%% file: aztecsquare.tex
\begin{figure}[h]
\begin{center}
\subfloat[The maximal filling shape for $A(4)$ of volume $30$]
	{\includegraphics[scale=0.18]{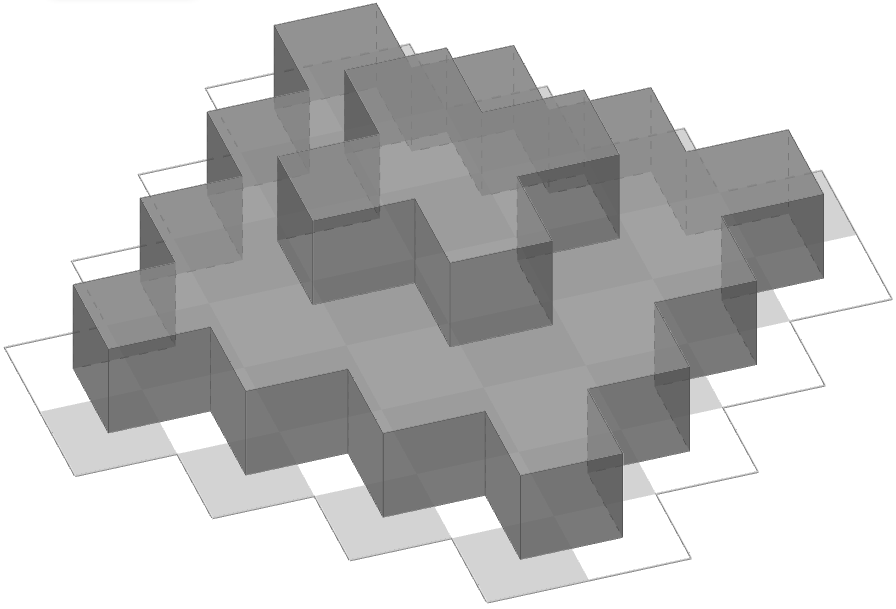}} \quad \quad
\subfloat[The maximal filling shape for $Q(6)$ of volume $35$]
	{\includegraphics[scale=0.16]{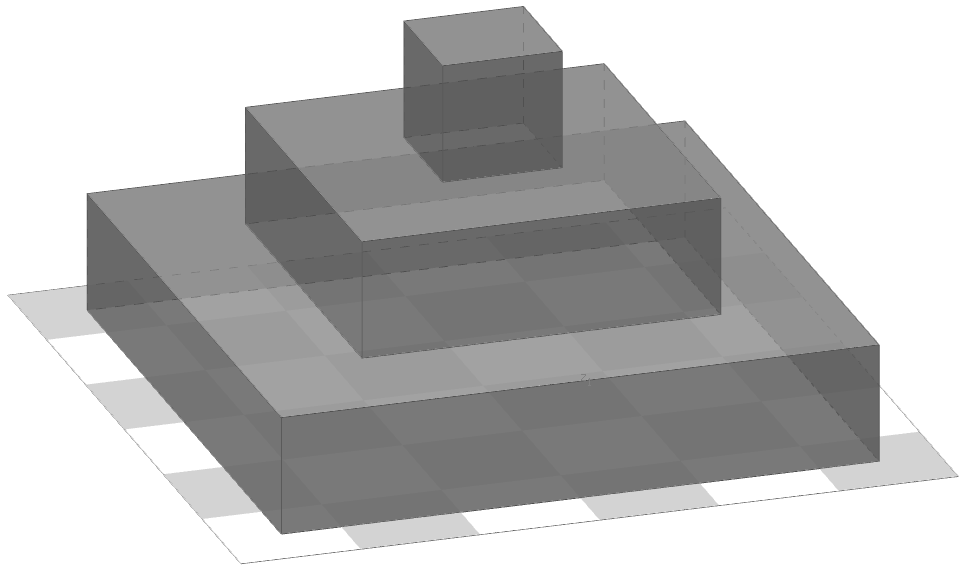}}
\caption{} 
\label{ghx:aztecsquare}
\end{center}
\end{figure}